\documentclass[11pt]{amsart}

\pdfoutput=1

\usepackage[text={420pt,660pt},centering]{geometry}

\usepackage{esint,amssymb} 
\usepackage{graphicx}
\usepackage{MnSymbol}
\usepackage{mathtools} 
\usepackage[colorlinks=true, pdfstartview=FitV, linkcolor=blue, citecolor=blue, urlcolor=blue,pagebackref=false]{hyperref}
\usepackage{microtype}

\usepackage{bm}
\usepackage{dsfont}

\newtheorem{proposition}{Proposition}
\newtheorem{theorem}[proposition]{Theorem}
\newtheorem{lemma}[proposition]{Lemma}

\theoremstyle{remark}
\newtheorem{remark}[proposition]{Remark}

\theoremstyle{definition}
\newtheorem{definition}[proposition]{Definition}

\numberwithin{equation}{section}
\numberwithin{proposition}{section}
\numberwithin{figure}{section}
\numberwithin{table}{section}

\newcommand{\N}{\mathbb{N}}
\newcommand{\Q}{\mathbb{Q}}
\newcommand{\R}{\mathbb{R}}

\newcommand{\E}{\mathbb{E}}
\renewcommand{\P}{\mathbb{P}}

\newcommand{\ep}{\varepsilon}

\renewcommand{\le}{\leqslant}
\renewcommand{\ge}{\geqslant}

\renewcommand{\subset}{\subseteq}
\newcommand{\la}{\left\langle}
\newcommand{\ra}{\right\rangle}

\newcommand{\Ll}{\left}
\newcommand{\Rr}{\right}
\renewcommand{\d}{\mathrm{d}}

\DeclareMathOperator{\tr}{tr}

\DeclareMathOperator{\supp}{supp}

\renewcommand{\bar}{\overline}

\newcommand{\td}{\widetilde}
\renewcommand{\hat}{\widehat}

\newcommand{\1}{\mathds{1}}

\newcommand{\mcl}{\mathcal}
\newcommand{\msf}{\mathsf}

\newcommand{\al}{\alpha}
\newcommand{\be}{\beta}

\newcommand{\de}{\delta}
\newcommand{\si}{\sigma}
\newcommand{\A}{{\mathcal{A}}}

\renewcommand{\H}{\mathsf{H}}
\newcommand{\M}{\mathbb{M}}

\newcommand{\dr}{\partial}
\newcommand{\n}{\mathbf{n}}

\newcommand{\upa}{\uparrow}

\newcommand{\vb}{\, \big \vert \, }

\begin{document}

\author[J.-C. Mourrat]{Jean-Christophe Mourrat}
\address[J.-C. Mourrat]{Courant Institute of Mathematical Sciences, New York University, New York NY, USA; Ecole Normale Sup\'erieure de Lyon and CNRS, Lyon, France}
\email{jean-christophe.mourrat@ens-lyon.fr}

\date{\today}

\title[Free energy upper bound for mean-field vector spin glasses]{Free energy upper bound for mean-field 
\\
vector spin glasses}

\begin{abstract}
We consider vector spin glasses whose energy function is a Gaussian random field with covariance given in terms of the matrix of scalar products. For essentially any model in this class, we give an upper bound for the limit free energy, which is expected to be sharp. The bound is expressed in terms of an infinite-dimensional Hamilton-Jacobi equation. 
\end{abstract}

\maketitle

%
%
%
%
%
%

\section{Introduction}
\label{s.intro}

\subsection{Statement of the main result}
The goal of this paper is to prove an upper bound on the free energy of essentially arbitrary  mean-field fully connected vector spin glasses. The class of models studied here covers multi-type and vector spins, allows for the inclusion of variables coming from Poisson-Dirichlet cascades, and makes minimal assumptions on the reference probability measure for the spins. Let $D \ge 1$ be an integer that will be kept fixed throughout the paper, $\xi$ be a locally Lipschitz function defined on the set $\R^{D\times D}$ of $D$-by-$D$ matrices, $(\mcl H_N)_{N \in \N}$ be a sequence of finite-dimensional Hilbert spaces, and suppose that, for each $N \in \N$, there exists a centered Gaussian random field $(H_N(\sigma))_{\si \in \mcl H_N^D}$ with covariance structure given, for every $\si,\tau \in \mcl H_N^D$, by
\begin{equation}  
\label{e.def.cov}
\E \Ll[ H_N(\si) H_N(\tau) \Rr] = N \xi \Ll( \frac{\si \tau^*}{N} \Rr) ,
\end{equation}
where the notation $\si \tau^*$ denotes the matrix of scalar products 
\begin{equation}  
\label{e.def.matrix}
\si \tau^* = \Ll( \si_d \cdot \tau_{d'} \Rr)_{1 \le d,d' \le D}.
\end{equation}
We also give ourselves, for each $N \in \N$, a ``reference'' probability measure $P_N$ on $\mcl H_N^D$, and we assume that
\begin{equation}
\label{e.ass.PN}
\mbox{the support of $P_N$ is contained in the ball of $\mcl H_N^D$ of radius $\sqrt{N}$}.
\end{equation}
We understand the notion of ball (centered at the origin) with respect to the norm derived from the scalar product on $\mcl H_N^D$ given, for each $\si = (\si_1,\ldots, \si_D)$ and $\tau = (\tau_1,\ldots, \tau_D) \in \mcl H_N^D$, by
\begin{equation}  
\label{e.scal.prod.hnd}
\si \cdot \tau := \sum_{d = 1}^D \si_d \cdot \tau_d.
\end{equation}
The focus of the paper concerns the large-$N$ behavior of the free energy
\begin{equation}  
\label{e.vanilla.free.energy}
\frac 1 N \E \log \int \exp \Ll( H_N(\si) \Rr) \, \d P_N(\si).
\end{equation}
Identifying the limit of this quantity is a fundamental step towards developing a better understanding of the corresponding Gibbs measure. 

Naturally, the identification of the large-$N$ limit of the quantity in \eqref{e.vanilla.free.energy} is only possible if we assume some notion of asymptotic convergence for the reference measure $P_N$. The relevant notion of convergence will be explained in details below; in a nutshell, we are asking that a class of free energies associated with ``one-body'' (or ``non-interacting'') Hamiltonians converge as $N$ tends to infinity. A loose analogy would be to compare this with a requirement on the convergence of some Laplace transform of the measure; we refer to the actual requirement as concerning the convergence of the ``cascade transform'' of the measure $P_N$. In concrete examples, the measure $P_N$ is usually constructed as a product of low-dimensional measures, or as a product of uniform measures on high-dimensional spheres. In such cases, one can verify that the cascade transform of the measure $P_N$ converges as $N$ tends to infinity; see for instance \cite[Proposition~3.1]{parisi}.

In order to state the main result, we need to introduce some more notation. We denote by $S^D$ the space of symmetric matrices, and by $S^D_+$ and $S^D_{++}$ the susbsets of positive semidefinite and positive definite matrices respectively. For every $a,b \in \R^{D\times D}$, we denote the natural scalar product between $a$ and $b$ by $a \cdot b := \tr(a^* b)$, and $|a| := (a\cdot a)^\frac 1 2$, where $a^*$ is the transpose of $a$. For every metric space~$E$, we denote by $\mcl P(E)$ the space of Borel probability measures on $E$, and, for every $p \in [1,\infty]$, by $\mcl P_p(E)$ the subspace of $\mcl P(E)$ of probability measures with finite $p$-th moment. We write $\de_x$ for the Dirac probability measure at $x \in E$. We say that a measure $\mu \in \mcl P(S^D_+)$ is \emph{monotonically coupled}, or simply \emph{monotone}, if, letting $X$ and $X'$ be two independent random variables with law $\mu$, we have for every $a,b \in S^D_+$ that
\begin{equation}  
\label{e.def.mon.meas}
\P \Ll[ a \cdot X < a \cdot X' \ \text{ and } b \cdot X > b \cdot X' \Rr] = 0.
\end{equation}
We denote by $\mcl P^{\uparrow}(S^D_+)$ the set of such measures, and set $\mcl P^\uparrow_p(S^D_+) := \mcl P^\uparrow(S^D_+) \cap \mcl P_p(S^D_+)$.

For any function $g : A \to B$ over (possibly partially) ordered sets $A$ and $B$, we say that $g$ is \emph{increasing} (over $A$) if
\begin{equation*}  
\forall a,a' \in A, \quad a \le a' \quad \implies \quad g(a) \le g(a').
\end{equation*}
We say that a mapping $\bar \xi : S^D_+ \to \R$ is \emph{proper} if $\bar \xi$ is increasing over $S^D_+$, and for every $b \in S^D_+$, the mapping $a \mapsto \bar \xi(a+b) - \bar \xi(a)$ is increasing over $S^D_+$. Here and throughout, we understand that the partial order on $S^D$ is that defined by the convex cone~$S^D_+$; that is, for every  $q,q' \in S^D$, we have $q \le q'$ if and only if $q'-q \in S^D_+$. We say that a mapping $\bar \xi : S^D_+ \to \R$ is a \emph{regularization} of the mapping $\xi : \R^{D \times D} \to R$ appearing in~\eqref{e.def.cov} if (1) the mappings $\bar \xi$ and $\xi$ coincide on the subset of positive semidefinite matrices with entries in $[-1,1]$; (2) the mapping $\bar \xi$ is uniformly Lipschitz; and (3) the mapping $\bar \xi$ is proper. Here is the main result of the paper.
 
\begin{theorem}  
\label{t.main1}
Let $\bar \xi$ be a regularization of $\xi$, and assume that the cascade transform of the measure $P_N$ converges to the function $\psi : \mcl P^\uparrow_2(S^D_+) \to \R$ as $N$ tends to infinity, in the sense of Definition~\ref{def.cascade.transf} below. For every $t \ge 0$, we have
\begin{equation}  
\label{e.main1}
\liminf_{N \to \infty} -\frac 1 N \E \log \int \exp \Ll( \sqrt{2t} H_N(\si) - Nt\xi\Ll(\frac {\si \si^*}{N}\Rr) \Rr) \, \d P_N(\si) \ge f(t,\de_0),
\end{equation}
where $f = f(t,\mu) : \R_+ \times \mcl P_2^\uparrow(S^D_+)\to \R$ is the solution to
\begin{equation}  
\label{e.hj}
\Ll\{
\begin{aligned}
& \dr_t f - \int \bar \xi(\dr_\mu f) \ \d  \mu  = 0 & \quad \text{on } \ \R_+ \times \mcl P_2^\uparrow(S^D_+),
\\
& f(0,\cdot ) = \psi & \quad \text{on } \ \mcl P_2^\uparrow(S^D_+).
\end{aligned}
\Rr.
\end{equation}
\end{theorem}
In all likelihood, the solution to \eqref{e.hj} does not depend on the choice of the regularization~$\bar \xi$. This point is discussed more precisely in Remark~\ref{r.barxi} below. 
At first sight, the requirement that~$\bar \xi$ be proper seems to impose a constraint on the function~$\xi$ itself. However, we will show in Proposition~\ref{p.charact} and Remark~\ref{r.charact} that if a function $\xi$ satisfies~\eqref{e.def.cov} for some random field $(H_N(\si))_{\si \in \mcl H_N}$, and admits a power series expansion, then it must be proper. I do not know if the assumption that $\xi$ admits a power series expansion is necessary. For any locally Lipschitz and proper function~$\xi$, an explicit construction of a regularization will be given in Subsection~\ref{ss.regularization}.


The additional term $N t\xi \Ll( \frac{\si \si^*}{N} \Rr)$ in \eqref{e.main1} is natural, since it normalizes the exponential to have expectation equal to $1$, but it may be seen as a hindrance to gaining information about \eqref{e.vanilla.free.energy}. When $\xi$ is convex, this term can be removed a posteriori as explained in~\cite{HJsoft}. In the general case, one should be able to obtain a description of the limit of~\eqref{e.vanilla.free.energy} by adding a finite-dimensional variable into the equation \eqref{e.hj} (so as to describe the infinite-volume limit of the quantity in \cite[(1.12)]{HJsoft} with $s = t$ there). 

The quantity $\dr_\mu f$ appearing in \eqref{e.hj} is a transport-type derivative. We will make sense of the equation \eqref{e.hj} as the limit of finite-dimensional approximations, based on replacing the set of measures $\mcl P^\uparrow_2(S^D_+)$ by the set of measures that are made of at most $K$ Dirac masses, and then letting $K$ tend to infinity. The precise definition of $\dr_\mu f$ when $\mu$ is a sums of Dirac masses is relatively straightforward and will be described more precisely in~\eqref{e.def.dmu}. Informally, for a given measure $\mu \in \mcl P(S^D_+)$ of finite support and $x \in \supp \mu$, the quantity $\dr_\mu f(t,\mu,x)$ measures the linear response of $f(t,\mu)$ to a variation in the position of the Dirac mass at the position $x$. A more explicit writing of the integral in \eqref{e.hj} is
\begin{equation*}  
\int \bar \xi(\dr_\mu f) \, \d \mu = \int \bar \xi(\dr_\mu f(t,\mu,x)) \, \d \mu(x). 
\end{equation*}

In the case when $\xi$ is convex over $S^D_+$, the solution to \eqref{e.hj} can at least heuristically be rewritten using the Hopf-Lax formula. Under this restrictive assumption on $\xi$, it is very likely that the arguments in \cite{parisi,HJsoft} carry over and allow us to show that the formula thus derived for the quantity $f(t,\de_0)$ appearing on the right side of \eqref{e.main1} is indeed equal to the variational formulas obtained in earlier works, which we review next. 

\subsection{Previous works}

An important part of the literature focuses on the Sherrington-Kirkpatrick model, which corresponds to the case when $D = 1$, $\mcl H_N = \R^N$, $\xi(r) = r^2$, and $P_N$ is a product of Bernoulli $\pm 1$ random variables. In this context, fundamental insights were obtained in the physics literature using sophisticated non-rigorous techniques \cite{MPV}. This includes a proposed variational formula for the limit free energy of the model, now known as the Parisi formula \cite{parisi79,parisi80}. Rigorous justifications of this formula were later obtained in \cite{gue03, Tpaper,Tbook1,Tbook2, pan.aom,pan}, ultimately covering the more general class of $p$-spin models, for which $D = 1$, $\mcl H_N = \R^N$, and $\xi(r) = r^p$ for some integer $p \ge 2$. This was then further generalized in different directions, including by relaxing the requirement that the support of $P_N$ be a subset of a sphere \cite{pan05,pan.vec}, or covering certain cases in which spins have multiple types \cite{pan.multi}, or are vector-valued \cite{pan.potts,pan.vec}. 

As explained in Section~\ref{s.examples}, all these models can be represented in such a way that~\eqref{e.def.cov} holds. It is critical to the validity of the variational formulas proved in these works that the function $\xi$ be \emph{convex} over $S^D_+$. As argued in \cite[Section~6]{bipartite}, naive extensions of these formulas to models for which $\xi$ does not satisfy this property are false; and  there is at present no known variational formula that could serve as a plausible candidate for the limit free energy of such models. An example of a model which fails to satisfy this convexity requirement on $\xi$ is a bipartite model, in which the spins are organized in two different layers, and the only direct interactions are between spins in different layers. 

The present work is part of an ongoing effort to identify the limit free energy of spin glass models using a different point of view that connects it to the solution of a certain Hamilton-Jacobi equation. Heuristic connections between limit free energies and partial differential equations were first pointed out in \cite{gue01, barra1, abarra, barra2}, under a replica-symmetric or one-step replica symmetry breaking assumption. In a different setting that relates to statistical inference, the possibility to relate the limit free energy with Hamilton-Jacobi equations was demonstrated rigorously in \cite{HJinfer, HJrank,HB1,HBJ}. 
That there exists a connection between the Parisi formula and a Hamilton-Jacobi equation was discovered in~\cite{parisi}; see also~\cite{HJsoft}. 
The connection between the Parisi formula and the solution of the partial differential equation makes use of the Hopf-Lax representation of the solution. However, this Hopf-Lax representation is only valid under the convexity assumption on $\xi$.

The present work is a generalization of \cite{bipartite}, in which a result analogous to Theorem~\ref{t.main1} is proved for the bipartite model. As will be explained in Section~\ref{s.enriched}, see in particular Theorem~\ref{t.main} and \eqref{e.almost.equation}, the proof of Theorem~\ref{t.main1} rests crucially on the possibility to argue for the synchronization of certain overlaps. Overlaps are scalar products between different independent copies of the random variables of interest, sampled according the Gibbs measure for fixed disorder couplings; denoting by $\si, \si'$ two such copies, we will want to get a synchronization result for $\si \si'^*$. Compared with \cite{bipartite}, there are several new aspects that need to be taken care of. One is that we can only ``synchronize'' the symmetric part of $\si \si'^*$, so we need to argue separately that the antisymmetric part becomes asymptotically negligible. The other is that the description of ``synchronized'' random variables taking values in $S^D$ is more involved than in the setting of $\R^2$ that was explored in \cite{bipartite}. The arguments for synchronization and symmetrization employed here rely heavily on the approach developed in \cite{pan.multi, pan.potts, pan.vec} to this effect. These considerations rely themselves on the fundamental property of {ultrametricity} of the Gibbs measure, which was obtained in \cite{pan.aom}. Other works which utilize this synchronization mechanism, and which all pertain to the setting investigated here, include \cite{jakose,weikuo,conmin,adhbre,ko}.

Compared with \cite{bipartite}, the treatment of the boundary condition for \eqref{e.hj} also poses new difficulties. This relates to the fact that the geometry of the space $\mcl P^\upa(S^D_+)$ is more intricate than that of the space $\mcl P^\upa(\R^2_+)$ (or $\mcl P^\upa(\R^D_+)$) that was appearing in \cite{bipartite}. Roughly speaking, although we expect the solution to \eqref{e.hj} to be such that the mapping $x \mapsto \dr_{\mu} f(\mu,x)$ is nonnegative and increasing over the support of $\mu$, we will extend the nonlinearity $\dr_\mu f \mapsto \int \bar \xi (\dr_\mu f) \, \d \mu$ outside of this ``natural'' set; and the extension we choose is not what the notation $\int \bar \xi (\dr_\mu f) \, \d \mu$ would suggest. 

It would of course be desirable to prove the bound converse to \eqref{e.main1}. Having a general upper bound for the limit free energy, as obtained here, might be useful for this very purpose, in analogy with the original proof of the Parisi formula developed in \cite{Tpaper,Tbook2}. 

To conclude this subsection, we also point out that in the very recent works \cite{subag1,subag1.5,subag2}, a generalized TAP approach is developed that yields the identification of the limit free energy of any pure multi-species spherical $p$-spin model, under a modest assumption of convergence of the free energy. 

\subsection{Organization of the paper} In Section~\ref{s.monotone}, which can be read independently of the rest of the paper, we explore the notion of a monotone coupling for a random variable taking values in $S^D_+$. Returning to the problem at hand, in Section~\ref{s.enriched} we introduce an enriched version of the free energy that also depends on a measure $\mu \in \mcl P^\upa_1(S^D_+)$. A heuristic calculation concluded in \eqref{e.almost.equation} suggests that this enriched free energy is an approximate solution to \eqref{e.hj}. Theorem~\ref{t.main}, generalizing Theorem~\ref{t.main1},  states an inequality between the limit of the enriched free energy and the solution to \eqref{e.hj}. In Section~\ref{s.visc}, we clarify the meaning of \eqref{e.hj} using finite-dimensional approximations, and in particular, we explain rigorously how the nonlinearity $\dr_\mu f \mapsto \int \bar \xi (\dr_\mu f) \, \d \mu$ is extended outside of its ``natural'' domain. We then move on to the proof of Theorem~\ref{t.main} in Section~\ref{s.super}. Finally, Section~\ref{s.examples} describes a variety of models that can be represented in the form of \eqref{e.def.cov}, and verifies that for all these models, and in fact  for essentially every conceivable model, the corresponding function $\xi$ is proper, and admits a regularization.

%
%
%
%
%
%
\section{Monotone couplings}
\label{s.monotone}
In this section, we study the notion of monotone measures over $S^D_+$, which was introduced around \eqref{e.def.mon.meas}. We explore various characterizations of this property, including that every such random variable can be realized as the image of the uniform measure over $[0,1]$ by an increasing mapping, see Proposition~\ref{p.uparrow.bij} below. This is related to considerations appearing in \cite{pan.potts, pan.vec}; the main point here is to clarify what part of the arguments are in fact internal to the concept of a monotone coupling, and unrelated to spin-glass considerations. 

 We start by recalling the notion in the simpler setting of measures over $\R^2$. 
Let $(X,Y)$ be a pair of real random variables (with respect to the probability $\P$, with expectation $\E$). We say that the pair $(X,Y)$ is \emph{monotonically coupled}, or simply \emph{monotone}, if, with $(X',Y')$ denoting an independent copy of $(X,Y)$, we have
\begin{equation}  
\label{e.monotone.r2}
\P \Ll[ X < X' \ \text{ and } \ Y' < Y \Rr] = 0.
\end{equation}
We also say that a measure $\mu \in \mcl P(\R^2)$ is \emph{monotonically coupled}, or more simply \emph{monotone}, whenever a random pair $(X,Y)$ with law $\mu$ satisfies this property.

For any real random variable $X$ and $x \in \R$, we write
\begin{equation}  
\label{e.def.cdf}
F_X(x)  := \P \Ll[ X \le x \Rr],
\end{equation}
and, for every $u \in [0,1]$,
\begin{equation}  
\label{e.def.cdfinv}
F_X^{-1}(u) := \inf \Ll\{ x \in \R \ : \ \P \Ll[ X \le x \Rr] \ge u \Rr\}  \in \R \cup \{\pm \infty\} .
\end{equation}
The next proposition provides with  equivalent characterisations of monotone couplings, for real-valued random variables. 
\begin{proposition}[Monotone coupling in $\R^2$]
\label{p.mono.coupl.r2}
Let $(X,Y)$ be a pair of real random variables, and $U$ be a uniform random variable over $[0,1]$. The following three statements are equivalent.

(1) The pair $(X,Y)$ is monotonically coupled.

(2) For every $(x,y)$ in the support of the law of $(X,Y)$, we have
\begin{equation}  
\label{e.coupl.support}
\P \Ll[ X < x \ \text{ and } \  Y > y \Rr] = 0.
\end{equation}

(3) For every $x, y \in \R$, we have
\begin{equation}  
\label{e.mono.coupl.r2}
\P \Ll[ X \le x \ \text{ and }\  Y \le y \Rr] = \min \Ll( \P \Ll[ X \le x \Rr] , \P \Ll[ Y \le y \Rr]  \Rr) .
\end{equation}

(4) The law of the pair $(X,Y)$ is that of $(F_X^{-1}(U),F_Y^{-1}(U))$.

(5) The following relations hold almost surely
\begin{equation*}  
X = F_X^{-1}(F_{X+Y}(X+Y)) \ \text{ and } \ Y = F_Y^{-1}(F_{X+Y}(X+Y)).
\end{equation*}
\end{proposition}

\begin{proof}
The equivalence between the statements (1), (3) and (4) was shown in \cite[Proposition~5.2]{bipartite}. 

The fact that (2) implies (1) is clear. We now show the converse implication. We thus assume that Statement (1) holds, and let $(x,y)$ be in the support of the law of $(X,Y)$. For each $\ep > 0$, we have
\begin{equation*}  
\P \Ll[ X \ge x - \ep \ \text{ and }  \ Y \le y+\ep  \Rr] > 0.
\end{equation*}
By the definition of a monotone coupling, see \eqref{e.monotone.r2}, we infer that
\begin{equation*}  
\P\Ll[X < x - \ep \ \text{ and } \ Y > y + \ep \Rr] = 0.
\end{equation*}
Letting $\ep$ go to zero, we obtain \eqref{e.coupl.support}. 

One can check that (5) implies (1) using the monotonicity of $F_X$ and $F_X^{-1}$. Conversely, suppose that (1) holds. We notice that for every $u \in (0,1)$, the infimum in the definition of $F_X^{-1}$ is achieved, and therefore $F_X(F_X^{-1}(u)) \ge u$. We now use this inequality with $X$ substituted by $X+Y$. 
Observing also that $X$ and $X+Y$ are monotonically coupled, and that (4) holds, we deduce that, almost surely,
\begin{equation*}  
X \le F_X^{-1} (F_{X+Y}(X+Y)) .
\end{equation*}
To show the converse implication, we start by observing that, for every $(x,y)$ in the support of the law of $(X,Y)$, we have
\begin{equation}  
\label{e.ineq.xy}
\P \Ll[ X \le x \Rr] \ge \P[X+Y \le x + y].
\end{equation}
Indeed, we have by (2) (and the symmetry of the definition of monotone coupling) that
\begin{equation*}  
\P \Ll[ X > x \ \text{ and } \ Y < y \Rr]  = 0.
\end{equation*}
In other words, outside of an event of null measure, we have the implication
\begin{equation}  
\label{e.impl.coupl}
X > x \implies Y \ge y.
\end{equation}
Using the decomposition
\begin{align*}  
\P \Ll[ X + Y \le x + y  \Rr] \le \P \Ll[ X \le x \Rr]  + \P \Ll[ X + Y \le x + y \ \text{ and } \ X > x \Rr],
\end{align*}
and noticing that, by \eqref{e.impl.coupl}, the last probability must be zero, we obtain \eqref{e.ineq.xy}. As a consequence,
\begin{equation*}  
F_X^{-1}(F_{X+Y}(x+y)) \le F_X^{-1}(F_X(x)) \le x.
\end{equation*}
Since this inequality is valid for every $(x,y)$ in the support of the law of $(X,Y)$, this completes the proof.
\end{proof}
We also record the following observations, which will be convenient later on.
\begin{lemma}
\label{l.relation.fxy}
Under the assumptions of Proposition~\ref{p.mono.coupl.r2}, we have
\begin{equation}  
\label{e.relation.fxy1}
F_{X+Y}^{-1}(U) = F_X^{-1}(U) + F_Y^{-1}(U) \qquad \text{a.s.}
\end{equation}
and
\begin{equation}  
\label{e.relation.fxy2}
F_X^{-1}(U) = F_X^{-1}(F_{X+Y}(F_{X+Y}^{-1}(U))) \qquad \text{a.s.}
\end{equation}
\end{lemma}
\begin{proof}
The relation \eqref{e.relation.fxy2} follows from the fact that $X$ and $X+Y$ are monotonically coupled, and parts (4) and (5) of Proposition~\ref{p.mono.coupl.r2}. Turning to the proof of \eqref{e.relation.fxy1}, we start by observing that, for any random variable $Z$, we have 
\begin{equation}  
\label{e.compose.FX}
Z = F_X^{-1}(F_X(Z)) \qquad \text{a.s.}
\end{equation}
This can be recovered from part (5) of Proposition~\ref{p.mono.coupl.r2}, using that $0$ and $Z$ are monotonically coupled. Using this property once more, we thus deduce that
\begin{equation*}  
F_{X+Y}^{-1}(F_{X+Y}(X+Y)) = F_{X}^{-1}(F_{X+Y}(X+Y)) + F_Y^{-1}(F_{X+Y}(X+Y)) \qquad \text{a.s.}
\end{equation*}
Defining
\begin{equation*}  
U_{X+Y} := F_{X+Y}(F_{X+Y}^{-1}(U)),
\end{equation*}
we have shown that 
\begin{equation*}  
F_{X+Y}^{-1}(U_{X+Y}) = F_X^{-1}(U_{X+Y}) + F_Y^{-1}(U_{X+Y}) \qquad \text{a.s.}
\end{equation*}
Using parts (4) and (5) of Proposition~\ref{p.mono.coupl.r2} once more, we see that
\begin{equation*}  
F_X^{-1}(U) = F_X^{-1} (U_{X+Y}) \qquad \text{a.s.}
\end{equation*}
as well as
\begin{equation*}  
F_{X+Y}^{-1}(U) = F_{X+Y}^{-1} (U_{X+Y}) \qquad \text{a.s.}
\end{equation*}
Combining the last three displays, we obtain \eqref{e.relation.fxy1}.
\end{proof}

Let $X$ be a random variable taking values in $S^D_+$. We say that $X$ is \emph{monotonically coupled}, or more simply \emph{monotone}, if for every $a,b \in S^D_+$, the pair $(a\cdot X, b \cdot X)$ is monotonically coupled. We also say that a measure $\mu \in \mcl P(S^D_+)$ is \emph{monotonically coupled}, or more simply \emph{monotone}, whenever a random variable with law $\mu$ satisfies this property. In this case, we write $\mu \in \mcl P^\uparrow(S^D_+)$. 

The next proposition gives equivalent characterizations of monotone couplings over~$S^D_+$. Recall the definitions in \eqref{e.def.cdf} and \eqref{e.def.cdfinv}.

\begin{proposition}[Monotone coupling in $S^D_+$]
\label{p.mono.coupl.sd}
Let $\mu \in \mcl P(S^D_+)$, $X$ be a random variable with law $\mu$, and $U$ be a uniform random variable over $[0,1]$. The following two statements are equivalent.

(1) The measure $\mu$ is monotonically coupled.

(2) There exists a random variable $X'$ with law $\mu$ such that, for every $a \in S^D_+$, 
\begin{equation}  
\label{e.couple.X'}
a \cdot X' = F_{a\cdot X}^{-1}(U) \qquad \text{a.s.}
\end{equation}

Under this circumstance, we have, for every $a \in S^D_+$, 
\begin{equation}
\label{e.mono.coupl.sd}
\P \Ll[ X \le a \Rr] = \inf_{b \in S^D_+} \P \Ll[ b \cdot X \le b \cdot a \Rr] .
\end{equation}
\end{proposition}
\begin{proof}
We decompose the proof into two steps.

\emph{Step 1.} By Proposition~\ref{p.mono.coupl.r2}, the fact that (2) implies (1) is clear. We now turn to the converse implication. We write $K := \frac{D(D+1)}{2}$, and let $(b_1,\ldots, b_K)$ be a family of elements of $S^D_+$ that form a basis of $S^D$. We can construct a random variable $X'$ taking values in~$S^D$ such that, for every $k \in \{1,\ldots,K\}$,
\begin{equation*}  
b_k \cdot X' = F^{-1}_{b_k\cdot X}(U).
\end{equation*}
Notice that for any random variable $Z$ and $\lambda >0$, we have $F^{-1}_{\lambda Z} = \lambda F^{-1}_Z$. Using also \eqref{e.relation.fxy1}, we deduce that the identity \eqref{e.couple.X'} holds for every $a \in S^D_+$. There remains to show that the random variables $X$ and $X'$ have the same law.

We first observe that, for every $a,b,b' \in S^D_+$ and $\lambda > 0$ satisfying $b = \lambda a + b'$, we have
\begin{equation}  
\label{e.aX.as.bX}
a \cdot X = F_{a \cdot X}^{-1} \Ll( F_{b \cdot X}(b \cdot X) \Rr)  \qquad \text{a.s.}
\end{equation}
Indeed, this follows using part (5) of Proposition~\ref{p.mono.coupl.r2} with $X$ and $Y$ there replaced by $\lambda a \cdot X$ and $b' \cdot X$ respectively, and using again that $F_{\lambda a \cdot X}^{-1} = \lambda F_{a \cdot X}^{-1}$. By the same argument combined with \eqref{e.relation.fxy2},  we also have
\begin{equation*}  
a \cdot X' = F_{a \cdot X}^{-1} \Ll( F_{b \cdot X}(b \cdot X') \Rr)  \qquad \text{a.s.}
\end{equation*}
In particular, setting $b := b_1 + \cdots + b_k$, we see that the vectors $(b_k \cdot X)_{1 \le k \le K}$ and $(b_k \cdot X')_{1 \le k \le K}$ can be obtained as images, under the same mapping, of the variables $b \cdot X$ and $b\cdot X'$ respectively. Since the latter random variables have the same law, we deduce that the random vectors have the same law as well. This implies that the laws of $X$ and $X'$ are the same (and in particular, that $X'$ takes values in $S^D_+$ with probability one), as desired.

\emph{Step 2.} To prepare for the proof of \eqref{e.mono.coupl.sd}, we start by showing that, for every sequence $(b_n)_{n \in \N}$ of elements of $S^D_+$ and $N \in \N$, we have
\begin{equation}  
\label{e.induct.min}
\P \Ll[ \forall n \in \{0,\ldots ,N\} , \ b_n \cdot X \le b_n \cdot a \Rr] = \min_{n \in \{0,\ldots, N\} } \P \Ll[  b_n \cdot X \le b_n \cdot a \Rr]  .
\end{equation}
We prove this by induction on $N$. The case $N = 0$ is obvious. Fixing $N \in \N$ and assuming that the identity \eqref{e.induct.min} holds, we now show that it also holds with $N$ replaced by $N+1$. By the induction hypothesis, there exists $n_0 \in \{0,\ldots, N\}$ such that, up to an event of null measure, the events
\begin{equation*}  
\forall n \in \{0,\ldots, N\} , \ b_n \cdot X \le b_n \cdot a  \quad \text{ and } \quad b_{n_0}\cdot X \le b_{n_0} \cdot a 
\end{equation*}
coincide. Combining this with the fact that $X$ is monotone and part (3) of Proposition~\ref{p.mono.coupl.r2}, we obtain
\begin{align*}  
\P \Ll[ \forall n \in \{0,\ldots, N+1\} , \ b_n \cdot X \le b_n \cdot a \Rr] 
& = \P \Ll[ b_{n_0}\cdot X \le b_{n_0} \cdot a  \ \text{ and } \ b_{N+1}\cdot X \le b_{N+1} \cdot a  \Rr] 
\\
& = \min \Ll( \P \Ll[ b_{n_0}\cdot X \le b_{n_0} \cdot a\Rr], \P \Ll[ b_{N+1}\cdot X \le b_{N+1} \cdot a  \Rr]  \Rr) 
\\
& = \min_{n \in \{0,\ldots, N+1\} } \P \Ll[  b_n \cdot X \le b_n \cdot a \Rr]  .
\end{align*}
This completes the induction argument.

We now show \eqref{e.mono.coupl.sd}. This identity with the equal sign replaced by ``$\le$'' is clearly valid. Conversely, recall first (see for instance \cite[Lemma~2.2]{HJrank}) that, for every $q \in S^D$, we have
\begin{equation} 
\label{e.charact.psd}
q \ge 0 \quad \iff \quad \forall b \in S^D_+ \ b \cdot q \ge 0 .
\end{equation}
Let $(b_n)_{n \in \N}$ be a sequence of elements of $S^D_+$ that is dense in $S^D_+$. We have
\begin{align*}  
\P \Ll[ X \le a \Rr]  
& = \P \Ll[ \forall b \in S^D_+ , \ b \cdot X \le b \cdot a \Rr] 
\\
& = \P \Ll[ \forall n \in \N , \ b_n \cdot X \le b_n \cdot a \Rr]
\\
& = \lim_{N \to \infty} \P \Ll[ \forall n \in \{0,\ldots, N\} , \ b_n \cdot X \le b_n \cdot a \Rr].
\end{align*}
By \eqref{e.induct.min}, the latter probability is of the form $\P \Ll[ b \cdot X \le b \cdot a \Rr]$ for some $b \in S^D_+$ (which depends on $N$). This yields the converse bound, and thus completes the argument.
\end{proof}
A slight variation of the argument above also gives the following alternative description of the set $\mcl P^\uparrow(S^D_+)$.
\begin{proposition}
\label{p.uparrow.bij}
Define
\begin{equation*}  
\mcl M := \Ll\{ M : [0,1) \to S^D_+ \ : \ M \text{ is right-continuous with left limits, and is increasing} \Rr\} ,
\end{equation*}
where we recall that ``$M$ is increasing'' means that, for every $u,v \in [0,1)$,
\begin{equation*}  
u \le v \quad \implies \quad M(u) \le M(v) . 
\end{equation*}
Recall also that $U$ denotes a uniform random variable over $[0,1]$. The mapping
\begin{equation}  
\label{e.uparrow.bij}
\Ll\{
\begin{array}{rcl}
\mcl M & \to & \mcl P^\uparrow(S^D_+) \\
M & \mapsto & \mbox{law of $M(U)$}
\end{array}
\Rr.
\end{equation}
is a bijection. We denote the preimage of the measure $\mu \in \mcl P^\uparrow(S^D_+)$ by $M_\mu$.
\end{proposition}
\begin{proof}
One can check that the mapping in \eqref{e.uparrow.bij} is well-defined and injective. Using the notation of Proposition~\ref{p.mono.coupl.sd} and Step 1 of its proof, we now slightly rephrase the construction of the random variable $X'$ appearing there and show that the mapping in~\eqref{e.uparrow.bij} is surjective. For each $u \in (0,1)$, there exists a unique matrix $M^-(u)$ such that for every $k \in \{1,\ldots, K\}$, we have
\begin{equation*}  
b_k \cdot M^-(u) = F_{b_k \cdot X}^{-1}(u).
\end{equation*}
The mapping $u \mapsto M^-(u)$ is left-continuous with right limits on $(0,1)$. Moreover, we have shown that $M^-(U)$ takes values in $S^D_+$ with probability one, and that for every $a \in S^D_+$, 
\begin{equation*}  
a \cdot M^-(U) = F_{a\cdot X}^{-1}(U) \qquad \text{a.s.}
\end{equation*}
Since both sides of this identity are left-continuous with right limits, we deduce that $M^-$ takes values in $S^D_+$ and that, for every $u \in (0,1)$ and $a \in S^D_+$, we have
\begin{equation*}  
a \cdot M^-(u) = F_{a\cdot X}^{-1}(u).
\end{equation*}
In particular, for every $a \in S^D_+$, we have that the mapping $u \mapsto a \cdot M(u)$ is increasing over $(0,1)$. This implies that the mapping $u \mapsto M^-(u)$ is increasing over $(0,1)$. Summarizing, we have identified, for each $\mu \in \mcl P^\upa(S^D_+)$, a mapping $M^- : (0,1) \to S^D_+$ which is increasing, left-continuous with right limits, and such that the law of $M^-(U)$ is $\mu$. For each $u \in [0,1)$, we set
\begin{equation*}  
M(u) := \lim_{v \downarrow u} M^-(v).
\end{equation*}
Since $M^-$ takes values in $S^D_+$, this quantity is well-defined and takes values in $S^D_+$. It is also increasing, and right-continuous with left limits at every point of $[0,1)$. Finally, since $M^-$ has only a countable number of points of discontinuity, we also have that the law of $M(U)$ is $\mu$. This completes the proof.
\end{proof}
Notice that the mappings $M_\mu$ allow us to define a joint coupling over all the measures in the set $\mcl P^\uparrow(S^D_+)$. Although we will not really need this fact, we observe in the next proposition that these are the optimal transport couplings for a very large class of cost functions, in analogy with the one-dimensional setting.
\begin{proposition}[Optimal transport in $\mcl P^\uparrow(S^D_+)$]
\label{p.transport}
Let $c : S^D_+ \times S^D_+ \to \R_+$ be a right-continuous function satisfying, for every $x,x',y,y' \in S^D_+$,
\begin{equation}  
\label{e.Monge.condition}
x \le x' \ \text{ and } \ y \le y' \quad \implies \quad c(x',y') + c(x,y) \le c(x,y') + c(x',y).
\end{equation}
Let $(X,Y)$ be a pair of random variables such that the law of $X$ is $\mu \in \mcl P^\uparrow(S^D_+)$ and the law of $Y$ is $\nu \in \mcl P^\uparrow(S^D_+)$. If 
\begin{equation*}  
\int c(x,y) \, \d \mu(x) \, \d \nu(y) < \infty,
\end{equation*}
then 
\begin{equation}  
\label{e.transport}
\E \Ll[ c(M_\mu(U), M_\nu(U)) \Rr] \le \E \Ll[ c(X,Y) \Rr] .
\end{equation}
\end{proposition}
\begin{remark}  
In Proposition~\ref{p.transport}, the statement that $c$ is right-continuous means that for every two sequences $(x_n)$ and $(y_n)$ of elements of $S^D_+$ satisfying, for every $n$, the inequalities $x_n \ge x$ and $y_n \ge y$, and such that $x_n \to x$ and $y_n \to y$ as $n$ tends to infinity, we have $c(x_n,y_n) \to c(x,y)$. Examples of functions satisfying the condition \eqref{e.Monge.condition} include any convex function of $x-y$. 
\end{remark}
\begin{proof}[Proof of Proposition~\ref{p.transport}]
For each $\mu \in \mcl P^\uparrow(S^D_+)$, we denote
\begin{equation*}  
M_\mu^{-1} :
\Ll\{
\begin{array}{rcl}  
\supp \mu & \to & [0,1] \\
x & \mapsto & \sup \Ll\{ u \in [0,1) \ : \ M_\mu(u) \le x  \Rr\} .
\end{array}
\Rr.
\end{equation*}
Since $\supp \mu$ is the closure of the image of $M_\mu$, the set appearing in the definition of $M_\mu^{-1}$ is not empty. The mapping $M_\mu^{-1}$ is right-continuous with left limits. Informally, it plays the role of a cumulative distribution function, similarly to the functions $F_X$ in one dimension. Letting $(X,Y)$ be as in the statement of the proposition, we have, similarly to \eqref{e.compose.FX},
\begin{equation}  
\label{e.compose.Mmu}
X = M_\mu(M_\mu^{-1}(X)) \qquad \text{a.s.}
\end{equation}
(This is obtained by noting that, for every $x \in \supp \mu$, we have $M_\mu (M_\mu^{-1}(x)) \le x$, while for every $u \in [0,1)$, we have $M_\mu^{-1}(M_\mu(u)) \ge u$.)
In particular,
\begin{equation*}  
\E \Ll[ c(X,Y) \Rr] = \E \Ll[ c\Ll(M_\mu(M_\mu^{-1}(X)), M_\nu(M_\nu^{-1}(Y))\Rr) \Rr] .
\end{equation*}
We consider the optimal transport problem on $[0,1]$ with cost function given, for every $u,v \in [0,1]$, by
\begin{equation*}  
\td c(u,v) := c \Ll( M_\mu(u), M_\nu(v) \Rr) .
\end{equation*}
By the property \eqref{e.Monge.condition} and the monotonicity of the mappings $M_\mu$ and $M_\nu$, we have, for every $u,u',v,v' \in [0,1]$,
\begin{equation*}  
u \le u' \ \text{ and } \ v \le v' \quad \implies \quad \td c(u',v') + \td c(u,v) \le \td c(u,v') + \td c(u',v).
\end{equation*}
As a consequence (see for instance \cite[Theorem~3.1.2]{racrus}), an optimal transport for the cost function $\td c$ with marginals given by the laws of $M_\mu^{-1}(X)$ and $M_\nu^{-1}(Y)$ respectively is achieved by the monotone coupling between these laws. This monotone coupling is realized by the pair $(M_\mu^{-1}(M_\mu(U)), M_\nu^{-1}(M_\nu(U)))$ (recall from Proposition~\ref{p.mono.coupl.r2} that to verify this, it suffices to observe that this pair is monotonically coupled and has the correct marginals). That is,
\begin{equation*}  
\E \Ll[ c \Ll( M_\mu(M_\mu^{-1}(M_\mu(U))), M_\nu(M_\nu^{-1}(M_\nu(U))) \Rr)  \Rr] 
\le \E \Ll[ c\Ll(M_\mu(M_\mu^{-1}(X)), M_\nu(M_\nu^{-1}(Y))\Rr) \Rr] .
\end{equation*}
Using \eqref{e.compose.Mmu} once more (with $X$ replaced by $M_\mu(U)$), we obtain \eqref{e.transport}.
\end{proof}

%
%
%
%
%
%

\section{Enriched free energy}
\label{s.enriched}

In this section, we define the free energy of an enriched model. Roughly speaking, the enriched model is defined in terms of an energy function which, in addition to the original function $H_N$, also contains another term which can be interpreted as the energy associated with a random magnetic field with an ultrametric structure. 
This structure is parametrized by a measure $\mu \in \mcl P^\uparrow_1(S^D_+)$. 

It is much more convenient to start by defining this quantity under the additional assumption that the measure $\mu$ has finite support, and then argue by density. 
Using~\eqref{e.charact.psd}, one can check that a measure $\mu \in \mcl P^\uparrow(S^D_+)$ with finite support must be of the form
\begin{equation}  
\label{e.def.mu.de}
\mu = \sum_{k = 0}^K (\zeta_{k+1} - \zeta_k) \de_{q_k},
\end{equation}
with $K \in \N$, $\zeta_0,\ldots, \zeta_{K+1} \in \R$ satisfying
\begin{equation}  
\label{e.strict.zeta}
0 = \zeta_0 < \zeta_1 < \ldots < \zeta_{K+1} = 1,
\end{equation}
and $q_{-1},q_0,\ldots, q_K \in S^D_+$ satisfying
\begin{equation}  
\label{e.strict.q}
0 = q_{-1} \le q_0 \le q_1 \le \ldots \le q_{K-1} \le q_K \ \text{ and } \ \forall k \in \{1,\ldots, K\}, q_{k-1} \neq q_k.
\end{equation}
We now very briefly recall some properties of Poisson-Dirichlet cascades, and refer to \cite[(2.46)]{pan} for more precision. 
We denote the rooted tree with (countably) infinite degree and depth $K$ by
\begin{equation}  
\label{e.def.mclA}
\mcl A := \N^{0} \cup \N \cup \N^2 \cup \cdots \cup \N^K,
\end{equation}
where $\N^{0} = \{\emptyset\}$, and $\emptyset$ represents the root of the tree. For every $k \in \{0,\ldots, K\}$ and $\alpha \in \N^k \subset \mcl A$, we write $|\alpha| := k$ to denote the depth of the vertex $\alpha$ in the tree $\mcl A$. For every leaf $\alpha = (n_1,\ldots,n_K)\in \N^K$ and $k \in \{0,\ldots, K\}$, we write
\begin{equation}  
\label{e.def.truncate}
\alpha_{| k} := (n_1,\ldots, n_k),
\end{equation}
with the understanding that $\alpha_{| 0} = \emptyset$. For every $\al, \al' \in \N^K$, we write
\begin{equation*}  
\al \wedge \al' := \sup \Ll\{ k \le K \ : \ \al_{|k} = \al'_{|k} \Rr\} .
\end{equation*}
A Poisson-Dirichlet cascade $(v_\al)_{\al \in \N^K}$ can be interpreted as a probability measure on the set $\N^K$ of leaves of the tree $\mcl A$, since $v_\alpha \ge 0$ and $\sum_{\al \in \N^K} v_\alpha = 1$; it is meant to serve as a ``canonical'' ultrametric structure (associated with the parameters $(\zeta_k)_{1 \le k \le K}$).  
Briefly, the construction of $(v_\alpha)_{\al \in \N^K}$ for~$K \ge 1$ is as follows: the progeny of each non-leaf vertex at level $k \in \{0,\ldots, K-1\}$ is decorated with the values of an independent Poisson point process of intensity measure $\zeta_{k+1} x^{-1-\zeta_{k+1}} \, \d x$, then the weight of a given leaf $\alpha \in \N^K$ is calculated by taking the product of the ``decorations'' attached to each parent vertex, including the leaf vertex itself (but excluding the root, which has no assigned ``decoration''), and finally, these weights over leaves are normalized so that their total sum is~$1$. We take this Poisson-Dirichlet cascade $(v_\alpha)_{\alpha \in \N^K}$ to be independent of $H_N$. 

Since we need to define some refined (ultrametric) analogue of an extraneous random magnetic field, we need to also introduce some additional random Gaussian variables that will play such a role. By definition, a standard Gaussian vector $z$ over a finite-dimensional Hilbert space $\mcl H$ is a random vector taking values in $\mcl H$ and such that, for every $\si, \si' \in \mcl H$, the random variable $\si \cdot z$ is a centered Gaussian and
\begin{equation}  
\label{e.standard.Gaussian}
\E \Ll[ (\si \cdot z)(\si' \cdot z) \Rr] = \si \cdot \si'.
\end{equation}
An explicit construction of $z$ can be obtained by selecting an orthonormal basis $(e_i)_{i \in I}$ of~$\mcl H$, letting $(z_i)_{i \in I}$ be independent standard Gaussians, and setting $z := \sum_{i \in I} z_i e_i$. 
We give ourselves $(z_\al)_{\al \in \mcl A}$ an independent family of standard Gaussian vectors over $\mcl H_N^D$. We take this family to be independent of $H_N$ and of the Poisson-Dirichlet cascade.
For every $\si \in \mcl H_N^D$ and $\al \in \N^K$, we set 
\begin{equation}  
\label{e.def.Hmu}
H_N^\mu(\si,\al) := \sum_{k = 0}^K \Ll( 2 q_k - 2 q_{k-1} \Rr) ^\frac 1 2 z_{\al_{|k}} \cdot \si.  
\end{equation}
Recall that $2 q_k - 2 q_{k-1} \in S^D_+$, and thus the square root of this matrix is well-defined. Any $D$-by-$D$ matrix $a = (a_{d,d'})_{1 \le d,d' \le D}$ acts on $\mcl H_N^D$ according to
\begin{equation*}  
a \sigma = \Ll(\sum_{d' =1 }^D a_{d,d'} \si_{d'}\Rr)_{1 \le d \le D}.
\end{equation*}
It is with this understanding that the expression $\Ll( 2 q_k - 2 q_{k-1} \Rr) ^\frac 1 2 z_{\al_{|k}}$ in \eqref{e.def.Hmu} is understood. (This is nothing but a convenient way to describe a random Gaussian vector with a particular covariance structure depending on the matrix $2q_k - 2q_{k-1}$.) In \eqref{e.def.Hmu}, the resulting element of $\mcl H_N^D$ is then ``dotted'' against $\si$, according to the scalar product in $\mcl H_N^D$ defined in \eqref{e.scal.prod.hnd}. 
Recall that $(H_N(\si))_{\si \in \mcl H_N^D}$ is a centered Gaussian field with covariance given in \eqref{e.def.cov}. For every $t \ge 0$ and $\mu$ as in \eqref{e.def.mu.de}, we define the enriched free energy as
\begin{multline}
\label{e.def.FN}
F_N(t,\mu) 
\\
:= -\frac 1 N \log \int \sum_{\al \in \N^K} \exp \Ll( \sqrt{2t} H_N(\si) - Nt\xi\Ll(\frac{\si \si^*}{N}\Rr) + H_N^\mu(\si,\al) - \si\cdot q_K \si \Rr) \, v_\al \, \d P_N(\si).
\end{multline}
We denote by $\td \E$ the expectation with respect to the randomness coming from the Poisson-Dirichlet cascade $(v_\al)_{\al \in \N^K}$ and the random Gaussian fields $(z_\al)_{\al \in \mcl A}$. Since the only additional source of randomness in the problem comes from the Gaussian field $(H_N(\si))_{\si \in \mcl H_N^D}$, which is independent of $(v_\al)_{\al \in \N^K}$ and $(z_\al)_{\al \in \mcl A}$, the expectation $\td \E$ could alternatively be written as the conditional expectation with respect to $(H_N(\si))_{\si \in \mcl H_N^D}$:
\begin{equation*}  
\td \E \Ll[ \cdot  \Rr] = \E \Ll[ \cdot \vb (H_N(\si))_{\si \in \mcl H_N^D} \Rr] .
\end{equation*}
We define the (partially and fully) averaged free energies
\begin{equation}  
\label{e.def.barFN}
\td F_N(t,\mu) = \td \E \Ll[ F_N(t,\mu)\Rr] \quad \text{ and } \quad \bar F_N(t,\mu) := \E \Ll[ F_N(t,\mu) \Rr] .
\end{equation}
Associated with the partition function $F_N$ is a Gibbs measure, which we denote by $\la \cdot \ra$, with canonical random variable $(\si,\al)$ taking values in $\mcl H_N \times \N^K$. That is, for every bounded measurable function $f :\mcl H_N \times \N^K \to \R$, we have
\begin{multline}  
\label{e.def.Gibbs}
\la f(\si,\al) \ra = \exp (N F_N(t,\mu))
\\
\times  \int \sum_{\al \in \N^K} f(\si,\al) \exp \Ll( \sqrt{2t} H_N(\si) - Nt\xi\Ll(\frac{\si \si^*}{N}\Rr) + H_N^\mu(\si,\al) - \si\cdot q_K \si \Rr) \, v_\al \, \d P_N(\si).
\end{multline}
The measure $\la \cdot \ra$ depends on the choice of the parameters $t$ and $\mu$, although the notation leaves this implicit. We also consider independent copies of the random pair $(\si,\al)$ under~$\la \cdot \ra$, which are often called ``replicas'' and which we write $(\si',\al')$, $(\si'',\al'')$, etc., or, if more replicas are desired, $(\si^\ell, \al^\ell)_{\ell \in \N}$. 
We recall from \cite{Tbook2}, \cite[(2.82)]{pan}, or \cite[Lemma~2.3]{parisi} that, for every $k \in \{0,\ldots,K\}$, we have
\begin{equation}
\label{e.overlap.pdc}
\td \E \la \1_{\{\al \wedge \al' = k\}} \ra = \zeta_{k + 1} - \zeta_k. 
\end{equation}

The next proposition, which is similar to a result in \cite{gue03}, gives a continuity estimate for~$\td F_N(t,\cdot)$. In order to state it, we will appeal to the characterization of monotone measures obtained in Proposition~\ref{p.uparrow.bij}. Throughout the paper, we denote by $U$ a uniform random variable over $[0,1]$, which for notational convenience we assume to be also defined with respect to the probability measure $\P$.
Using the notation introduced in Proposition~\ref{p.uparrow.bij}, we define the random variable
\begin{equation}
\label{e.def.Xmu}
X_\mu := M_\mu(U).
\end{equation}
This construction provides us with a joint coupling of all the measures in $\mcl P^\uparrow(S^D_+)$, and thus allows us to write down optimal-transport-type distances in $\mcl P^\uparrow(S^D_+)$ in a convenient way (see Proposition~\ref{p.transport}). 
\begin{proposition}[Lipschitz continuity of $\td F_N$]
\label{p.lip}
For any two pairs $\mu, \nu \in \mcl P^\upa(S^D_+)$ of measures with finite support and $t \ge 0$, we have
\begin{equation}  
\label{e.lip}
|\td F_N(t,\mu) - \td F_N(t,\nu)| \le \E \Ll[ |X_\mu - X_\nu| \Rr] ,
\end{equation}
and the same inequality also holds with $\td F_N$ replaced by $\bar F_N$. In particular, the functions $\td F_N$ and $\bar F_N$ can be extended by continuity to $\R_+ \times \mcl P^\upa_1(S^D_+)$. 
\end{proposition}
Recall that throughout the paper, whenever $a = (a_{d,d'})_{1 \le d,d' \le D}$ and $b = (b_{d,d'})_{1 \le d,d' \le D}$ are two $D$-by-$D$ matrices with real entries, we use the notation
\begin{equation*}  
a \cdot b := \tr (a b^*) = \sum_{d,d' = 1}^D a_{d,d'} b_{d,d'}, \quad \text{ and } \quad |a|  := (a \cdot a)^\frac 1 2,
\end{equation*}
with $b^*$ denoting the transpose of $b$, and $\tr$ denoting the trace operator. 

For $\si \in \mcl H_N^D$, so far we have only made sense of the notation $\si^*$ within the notation ``$\tau \si^*$'', with $\si, \tau \in \mcl H_N^D$, in which case we recall that this is interpreted as the $D$-by-$D$ matrix
\begin{equation}  
\label{e.def.tausi*}
\tau \si^* := (\tau_d \cdot \si_{d'})_{1 \le d,d' \le D}.
\end{equation}
We now point out that one can make sense of the object $\si^*$ itself, in a way which is consistent with the notation in \eqref{e.def.tausi*}. To start with, in the case when $\mcl H_N = \R^N$, we can view $\si$ and $\tau$ as $D$-by-$N$ matrices; with this interpretation, we can define $\si^*$ to be the $N$-by-$D$ matrix conjugate to $\si$, and the notation $\tau \si^*$ is consistent with the matrix product. In the general case, we can identify $\si$ with the linear map
\begin{equation*}  
\Ll\{
\begin{array}{rcl}  
\mcl H_N  & \to & \R^D \\
\tau & \mapsto & (\si_1 \cdot \tau, \ldots, \si_D \cdot \tau).
\end{array}
\Rr.
\end{equation*}
This mapping admits a dual mapping from $\R^D$ to $\mcl H_N$, which we can denote by $\si^*$. Explicitly, the mapping $\si^*$ is given by
\begin{equation}  
\label{e.sigma*}
\Ll\{
\begin{array}{rcl}  
\R^D & \to & \mcl H_N  \\
v = (v_1,\ldots, v_D) & \mapsto & v_1 \si_1 + \cdots + v_D \si_D.
\end{array}
\Rr.
\end{equation}
The composition $\tau \si^*$ is then a linear map from $\R^D$ to itself, whose matrix representation in the canonical basis is indeed \eqref{e.def.tausi*}.

In order to prove Proposition~\ref{p.lip}, we will study the derivatives of the function $\td F_N$ with respect to the matrices $q_\ell$ appearing in \eqref{e.def.mu.de}. For any function $f = f(q) : U \to \R$ defined on an open subset $U$ of $S^D$, we write $\dr_q f(q) \in S^D$ to denote the unique symmetric matrix such that, for every $a \in S^D$,
\begin{equation*}  
\lim_{\ep \to 0} \ep^{-1} \Ll( f(q + \ep a) - f(q) \Rr) = \dr_q f(q) \cdot a,
\end{equation*}
provided that the limit exists for every $a \in S^D$. In view of the definition of $F_N$, we will need to differentiate the matrix square root operator; we thus recall some of its properties. We denote by $D_{\sqrt{2q}}$ the differential of the mapping 
\begin{equation*}  
\Ll\{
\begin{array}{rcl}  
S^D_{++} & \to & S^D \\
q  & \mapsto & \sqrt{2q}. 
\end{array}
\Rr.
\end{equation*}
That is, for every $q \in S^D_{++}$ and $a \in S^D$,
\begin{equation*}  
D_{\sqrt{2q}}(a) = \lim_{\ep \to 0} \ep^{-1} \Ll( \sqrt{2(q+\ep a)} - \sqrt{2q} \Rr) .
\end{equation*}
As $\ep \to 0$, we have
\begin{multline*}  
2q+2\ep a = \Ll(\sqrt{2q+2\ep a}\Rr)^2 = \Ll( \sqrt{2q} + \ep D_{\sqrt{2q}}(a) + o(\ep) \Rr)^2
\\
 = 2q + \ep \Ll( \sqrt{2q} D_{\sqrt{2q}}(a) + D_{\sqrt{2q}}(a) \sqrt{2q}  \Rr) + o(\ep), 
\end{multline*}
and thus
\begin{equation}
\label{e.charact.sqrt}
\sqrt{2q} D_{\sqrt{2q}}(a) + D_{\sqrt{2q}}(a) \sqrt{2q}  = 2a.
\end{equation}
This property characterizes the matrix $D_{\sqrt{2q}}(a)$ among symmetric matrices, see for instance \cite[Lemma~3.1]{HJrank}. In particular, if $b \in S^D$, then we can use the symmetry of $b$ and the fact that the transpose of $\sqrt{2q} D_{\sqrt{2q}}(a)$ is $D_{\sqrt{2q}}(a)\sqrt{2q}$ to get that
\begin{equation}  
\label{e.sqrt.dot.sym}
\Ll(\sqrt{2q} D_{\sqrt{2q}}(a)\Rr) \cdot b = \Ll(D_{\sqrt{2q}}(a) \sqrt{2q} \Rr)  \cdot b = a \cdot b.
\end{equation}

\begin{proof}[Proof of Proposition~\ref{p.lip}]
We decompose the proof into two steps.

\emph{Step 1.} In this first step, we show that for every $k \in \{0,\ldots, K\}$, 
\begin{equation}  
\label{e.expr.drqell}
\dr_{q_k} \td F_N = \frac 1 N \td \E \la \1_{\{\al \wedge \al' = k\}} \, \si \si'^* \ra   .
\end{equation}
For every $k \in \{0,\ldots, K-1\}$ and $a \in S^D$, we have
\begin{equation*}  
a \cdot \dr_{q_k} F_N= -\frac 1 N\la D_{\sqrt{2q_k - 2q_{k-1}}}(a)  z_{\al_{|k}} \cdot \si - D_{\sqrt{2q_{k+1} - 2q_{k}}}(a) z_{\al_{|k+1}} \cdot \si \ra.
\end{equation*}
Notice also that, for every $\si,\si' \in \mcl H_N^D$ and $\al,\al' \in \N^K$, we have
\begin{align*}  
& \td \E \Ll[ \Ll(D_{\sqrt{2q_k - 2q_{k-1}}}(a)  z_{\al_{|k}} \cdot \si \Rr) H_N^\mu(\si',\al') \Rr] 
\\
& \qquad
= \td \E \Ll[ \Ll(z_{\al_{|k}} \cdot D_{\sqrt{2q_k - 2q_{k-1}}}(a) \si \Rr) \Ll( (2q_k - 2q_{k-1})^\frac 1 2 z_{\al'_{|k}} \cdot \si' \Rr) \Rr] 
\\
& \qquad 
= \1_{\{\al \wedge \al' \ge k\}} \, D_{\sqrt{2q_k - 2q_{k-1}}}(a) \si \cdot (2q_k - 2q_{k-1})^\frac 1 2 \si' .
\end{align*}
By Gaussian integration by parts (see for instance \cite[Lemma~A.1]{bipartite}), we deduce that
\begin{multline*}  
\td \E \la D_{\sqrt{2q_k - 2q_{k-1}}}(a)  z_{\al_{|k}} \cdot \si \ra = 
\td \E  \la D_{\sqrt{2q_k - 2q_{k-1}}}(a)(2q_k - 2q_{k-1})^\frac 1 2 \cdot \si \si^* \ra
\\
- \td \E \la \1_{\{\al \wedge \al' \ge k\}} \, D_{\sqrt{2q_k - 2q_{k-1}}}(a)(2q_k - 2q_{k-1})^\frac 1 2 \cdot \si \si'^* \ra .
\end{multline*}
Since the matrix $\si \si^*$ is symmetric, we have by \eqref{e.sqrt.dot.sym} that 
\begin{equation*}  
\td \E  \la D_{\sqrt{2q_k - 2q_{k-1}}}(a)(2q_k - 2q_{k-1})^\frac 1 2 \cdot \si \si^* \ra = \td \E  \la a \cdot \si \si^* \ra .
\end{equation*}
We now observe that the matrix $\la \1_{\{\al \wedge \al' \ge k\}} \si \si'^* \ra$ is also symmetric. Indeed,
\begin{align}  
\notag
\la \1_{\{\al \wedge \al' \ge k\}} \si \si'^* \ra
&  = \sum_{\beta \in \N^k} \la \1_{\{\al_{|k} = \beta \}} \1_{\{\al'_{|k} = \beta\}} \si \si'^* \ra
\\
\label{e.sym.1.sisi}
& = \sum_{\beta \in \N^k} \la \1_{\{\al_{|k} = \beta \}} \si \ra \la \1_{\{\al_{|k} = \beta \}} \si \ra^*,
\end{align}
and therefore,
\begin{equation*}  
\td \E \la D_{\sqrt{2q_k - 2q_{k-1}}}(a)  z_{\al_{|k}} \cdot \si \ra  = \td \E \la a \cdot \si \si ^* -  \1_{\{\al \wedge \al' \ge k\}} a \cdot \si \si'^* \ra .
\end{equation*}
Using this also with $k$ replaced by $k + 1$ yields that
\begin{equation*}  
a \cdot \dr_{q_{k}} \td F_N = \frac 1 N \td \E  \la  \1_{\{\al \wedge \al' \ge k\}}  a \cdot \si \si'^* - \1_{\{\al \wedge \al' \ge k+1\}} a \cdot \si \si'^*\ra  .
\end{equation*}
Recalling from \eqref{e.sym.1.sisi} that the matrix $\1_{\{\al \wedge \al' = k\}} \si \si'^*$ is symmetric, we obtain \eqref{e.expr.drqell} for every $k < K$. The case $k = K$ is similar. This shows in particular that the function $\td F_N$ is uniformly Lipschitz continuous in the variables $(q_k)_{0 \le k \le K}$. More precisely, since
\begin{align*}  
|\si \si'^*|^2 
 = \sum_{d,d' = 1}^D (\si_d \cdot \si'_{d'})^2
 \le \sum_{d,d' = 1}^D |\si_d|^2 \, |\si'_{d'}|^2
 = |\si|^2 \, |\si'|^2,
\end{align*}
and recalling the assumption \eqref{e.ass.PN} on the support of $P_N$, and also \eqref{e.overlap.pdc} and the sentence below it, we obtain that, for every $k \in \{0,\ldots, K\}$,
\begin{equation}  
\label{e.bound.drq}
|\dr_{q_k} \td F_N| \le  \td \E \la\1_{\{\al \wedge \al' = k\}} \ra = \zeta_{k + 1} - \zeta_k.
\end{equation}
\emph{Step 2.} We complete the proof. First, we observe that we can allow for repetitions in the parameters $(q_k)_{0 \le k \le K}$ appearing in \eqref{e.def.mu.de}. That is, our definition of $\td F_N$ would also make sense if we were to drop the requirement of \eqref{e.strict.q} and replace it instead by simply
\begin{equation}  
\label{e.wide.q}
0 = q_{-1} \le q_0 \le q_1 \le \cdots \le q_{K-1} \le q_K.
\end{equation}
We could also look for a representation of the measure $\mu$ in which the parameters $(q_k)_{0 \le k \le K}$ are not repeated, and this would formally lead to a different definition of $\td F_N$. The point we are making here is that these two quantities are indeed the same, as the notation $\td F_N(t,\mu)$ suggests. The argument to justify this is classical, see for instance Step 1 of the proof of \cite[Proposition~2.1]{parisi}. (In essence, the point is to observe that a Poisson-Dirichlet cascade of depth $K+1$ with nodes on level $k$ deleted has the same law as a Poisson-Dirichlet cascade of depth $K$ and in which the weight $\zeta_k$ is deleted.)

In view of this, given two measures $\mu, \nu \in \mcl P^\upa(S^D_+)$ with finite support, we can fix parameters $(\zeta_{k})_{1 \le k \le K}$ satisfying \eqref{e.strict.zeta}, and $q_0,\ldots, q_K,q'_0,\ldots, q'_K \in S^D_+$ satisfying \eqref{e.wide.q} and
\begin{equation}  
\label{e.wide.q'}
0 = q'_{-1} \le q'_0 \le q'_1 \le \cdots \le q'_{K-1} \le q'_K,
\end{equation}
in such a way that $\mu$ satisfies \eqref{e.def.mu.de}, while 
\begin{equation}  
\label{e.def.nu.de}
\nu = \sum_{k = 0}^K (\zeta_{k+1} - \zeta_k) \de_{q'_k}.
\end{equation}
In view of \eqref{e.bound.drq}, we obtain that
\begin{equation*}  
|\td F_N(t,\mu) - \td F_N(t,\nu)| \le \sum_{k = 0}^K (\zeta_{k+1} - \zeta_k) |q_k' - q_k|. 
\end{equation*}
This is \eqref{e.lip}.
\end{proof}
\begin{definition}  
\label{def.cascade.transf}
We call the mapping 
\begin{equation*}  
\Ll\{
\begin{array}{rcl}  
\mcl P^\upa_1(S^D_+) & \to & \R \\
\mu & \mapsto & \bar F_N(0,\mu)
\end{array}
\Rr.
\end{equation*}
the \emph{cascade transform} of the measure $P_N$. We say that the cascade transform of the measure $P_N$ converges to the function $\psi : \mcl P^\uparrow_2(S^D_+) \to \R$ as $N$ tends to infinity if, for every $\mu \in \mcl P^\uparrow_2(S^D_+)$, we have
\begin{equation*}  
\lim_{N \to \infty} \bar F_N(0,\mu) = \psi(\mu).
\end{equation*}
\end{definition}
\begin{remark}  
An alternative interpretation of the mapping  $\mu \mapsto \td F_N(t,\mu)$ is that it is the cascade transform of the measure
\begin{equation*}  
\exp \Ll( \sqrt{2t} H_N(\si) - Nt\xi\Ll(\frac{\si \si^*}{N}\Rr)  \Rr)\, \d P_N(\si).
\end{equation*}
(Normalizing this into a probability measure would only affect $\td F_N(t,\mu)$ by an additive constant.)
\end{remark}
The main focus of the paper is to prove the following result, which implies Theorem~\ref{t.main1}.
\begin{theorem}
\label{t.main}
Let $\bar \xi$ be a regularization of $\xi$, and assume that the cascade transform of the measure $P_N$ converges to the function $\psi : \mcl P^\uparrow_2(S^D_+) \to \R$ as $N$ tends to infinity. For every $t \ge 0$ and $\mu \in \mcl P_2^\upa(S^D_+)$, we have
\begin{equation*}  
\liminf_{N \to \infty} \bar F_N(t,\mu) \ge f(t,\mu),
\end{equation*}
where $f : \R_+ \times \mcl P_2^\uparrow(S^D_+) \to \R$ is the solution to \eqref{e.hj}. 
\end{theorem}
\begin{remark}  
\label{r.barxi}
Under the assumptions of Theorem~\ref{t.main}, I expect that $\bar F_N$ actually converges to $f$, and in particular, I expect that the solution to \eqref{e.hj} does not depend on the choice of~$\bar \xi$. To prove the latter point directly, the crucial missing ingredient is an extension of the estimate~\eqref{e.lip.hj.finite}, which ensures that a Lipschitz property of the initial condition is propagated to the whole solution. This estimate measures the norm of the gradient of the solution in an $\ell^2$ norm. What is missing is a version of this result in which this $\ell^2$ norm is replaced by an $\ell^\infty$ norm. 
\end{remark}
We now explain heuristically why one may expect the free energy $\bar F_N$ to converge to the function $f$ solution to \eqref{e.hj}. For starters, we clarify our definition of the transport derivative $\dr_\mu$. Informally, for any sufficiently smooth function $g : \mcl P^\uparrow_2(S^D_+) \to \R$ and $\mu \in \mcl P^\upa_2(S^D_+)$, we want to define a function $\dr_\mu g(\mu,\cdot) \in L^2(S^D_+,\mu)$ such that, as $\nu$ tends to~$\mu$ in $\mcl P_2(S^D_+)$,
\begin{equation*}  
g(\nu) = g(\mu) + \E \Ll[ \dr_\mu g(\mu,X_\mu)(X_{\nu} - X_\mu) \Rr]  + o \Ll( \E \Ll[ (X_{\nu} - X_\mu)^2 \Rr] ^\frac 1 2  \Rr) .
\end{equation*}
In practice, we will always work with measures of finite support, in which case we can rely on the following explicit definition. For any measure $\mu$ of the form \eqref{e.def.mu.de}-\eqref{e.strict.q}, and $k \in \{0,\ldots, K\}$, we set
\begin{equation}
\label{e.def.dmu}
\dr_\mu g (\mu,q_k) := (\zeta_{k+1} - \zeta_k)^{-1} \dr_{q_k} g(\mu),
\end{equation}
where on the right side of this identity, we interpret $\dr_{q_k} g(\mu)$ as the derivative with respect to $q_k$ of the function
\begin{equation*}  
(q_0,\ldots, q_K) \mapsto g \Ll( \sum_{\ell = 0}^K (\zeta_{\ell+1} - \zeta_{\ell}) \de_{q_{\ell}} \Rr) .
\end{equation*}
Coming back to the heuristic derivation of the equation \eqref{e.hj}, we first notice that
\begin{equation*}  
\dr_t F_N = -\frac 1 N \la \frac{1}{\sqrt{2t}} H_N(\si) - N \xi \Ll( \frac{\si \si^*}{N} \Rr)\ra ,
\end{equation*}
and thus, by Gaussian integration by parts (see for instance \cite[Lemma~A.1]{bipartite}),
\begin{equation}  
\label{e.expr.drt}
\dr_t \bar F_N = \E \la \xi \Ll( \frac{\si \si'^*}{N} \Rr) \ra. 
\end{equation}
On the other hand, it follows from \eqref{e.expr.drqell} that, for measures $\mu$ of the form given in \eqref{e.def.mu.de} and every $k \in \{0,\ldots, K\}$,
\begin{equation*}  
\dr_{q_k} \bar F_N = \frac 1 N \E \la \1_{\{\al \wedge \al' = k\}} \, \si \si'^* \ra.
\end{equation*}
Recalling also \eqref{e.overlap.pdc}, we can write
\begin{equation}  
\label{e.drqk.normalized}
(\zeta_{k+1} - \zeta_k)^{-1} \dr_{q_k} \bar F_N = \frac 1 N \E \la \si \si'^* \vb \al \wedge \al' = k \ra,
\end{equation}
where the conditional expectation is understood with respect to the measure $\E \la \cdot \ra$. Using \eqref{e.overlap.pdc} once more, we thus have
\begin{align*}  
\int \xi(\dr_\mu \bar F_N) \, \d \mu = \sum_{k = 0}^K (\zeta_{k+1} - \zeta_k) \xi \Ll(  \E \la \frac{\si \si'^*}{N} \vb \al \wedge \al' = k \ra \Rr)  = \E \la   \xi\Ll(\E \la \frac{\si \si'^*}{N} \vb \al \wedge \al' \ra \Rr)  \ra .
\end{align*}
Summarizing, we have shown that, at least for measures $\mu$ that are sums of Dirac masses,
\begin{equation}  
\label{e.almost.equation}
\dr_t \bar F_N - \int \xi(\dr_\mu \bar F_N) \, \d \mu = \E \la \xi \Ll( \frac{\si \si'^*}{N} \Rr) \ra - \E \la   \xi\Ll(\E \la \frac{\si \si'^*}{N} \vb \al \wedge \al' \ra \Rr)  \ra .
\end{equation}
We may call the matrix $\frac{\si \si'^*}{N}$ the $\si$-overlap, and $\al \wedge \al'$ the $\al$-overlap. The right side of~\eqref{e.almost.equation} is small if and only if the law of the $\si$-overlap given the $\al$-overlap is concentrated. That is, we need to assert a synchronization property, in the sense that the $\al$-overlap should essentially determine the $\si$-overlap. 

It would of course be ideal if one could show that the right side of \eqref{e.almost.equation} becomes small as $N$ becomes large, for any choice of the parameters $t$ and $\mu$. However, in all likelihood, the right-hand side of \eqref{e.almost.equation} will only be small for \emph{most} choices of the parameters. Indeed, even when studying much simpler situations such as that in which we add a small viscosity term to a simple Hamilton-Jacobi equation, the ``error term'' will typically not be uniformly small. The essence of the problem investigated in this paper revolves around such a difficulty: if we could assert that the right side of~\eqref{e.almost.equation} is small uniformly over the parameters, then we would immediately obtain the convergence of $\bar F_N$ to $f$ as a consequence of the comparison principle. However, this is too much to ask for. On the other hand, one can construct examples (even in $1+1$ dimensions) for which the ``error term'' in the equation is small in $L^1$ (in the ``space'' variable), and yet which do not converge to the viscosity solution of the equation without error terms. Hence, in order to conclude for the convergence of $\bar F_N$ to $f$, we need to identify a rather delicate mechanism by which the viscosity solution to the limit equation is selected. As far as the lower bound is concerned (that is, the content of Theorem~\ref{t.main}), this ``selection principle'' will rest on some concavity property of the function $\bar F_N$ with respect to additional perturbative parameters that are yet to be introduced. We refer to the discussion below \eqref{e.super} for more on this.

We conclude this section with some basic inequalities on the derivatives of $\td F_N$, which will be used as boundary conditions for the equation \eqref{e.hj}. 
\begin{proposition}
[Basic inequalities on $\dr_{q_k} \td F$]
\label{p.basic.ineq}
For $\mu$ a measure of the form \eqref{e.def.mu.de}, we have, for every $k \le \ell \in \{0,\ldots, K\}$,
\begin{equation}
\label{e.monotone}
\dr_{q_k} \td F_N \ge 0,
\end{equation}
as well as
\begin{equation}  
\label{e.monotone.cone}
(\zeta_{k+1} - \zeta_k)^{-1} \dr_{q_k} \td F_N \le (\zeta_{\ell+1} - \zeta_{\ell})^{-1} \dr_{q_{\ell}} \td F_N .
\end{equation}
\end{proposition}
Recall that the inequalities in \eqref{e.monotone} and \eqref{e.monotone.cone} are interpreted in the sense of the partial order on $S^D$. For instance, an equivalent formulation of \eqref{e.monotone} is to say that the symmetric matrix $\dr_{q_k} \td F_N$ is positive semidefinite. 
\begin{proof}[Proof of Proposition~\ref{p.basic.ineq}]
We start by introducing notation.
For each $y_0,\ldots, y_K \in \mcl H_N^D$, we define
\begin{multline*}  
X_K(y_0,\ldots, y_K) :=  \log \int  \exp \bigg( \sqrt{2t} H_N(\si) - Nt\xi\Ll(\frac{\si \si^*}{N}\Rr) 
\\
+  \sum_{k = 0}^K \Ll( 2 q_k - 2 q_{k-1} \Rr) ^\frac 1 2 y_k \cdot \si - \si\cdot q_K \si \bigg) \,  \d P_N(\si),
\end{multline*}
We then define recursively, for each $k \in \{1,\ldots, K\}$,
\begin{equation*}  
X_{k-1}(y_0,\ldots,y_{k-1}) := \zeta_k^{-1} \log \E_{y_k} \exp \Ll( \zeta_k X_k(y_0,\ldots, y_k) \Rr),
\end{equation*}
where we use the notation $\E_{y_k}$ to denote the integration of the random variable $y_k$ along the standard Gaussian measure on $\mcl H_N^D$,  see \eqref{e.standard.Gaussian}. We will also use below the notation $\E_{y \ge k}$ to denote the integration of the variables $y_k,\ldots, y_K$ along the standard Gaussian measure, and use the shorthand $\E_y$ for $\E_{y \ge 0}$.
By \cite[Proposition~2.2]{bipartite}, we have
\begin{equation}  
\label{e.tdF.is.X-1}
-N\td F_N(t,\mu) = X_{-1} := \E_{y_0}\Ll[ X_0(y_0) \Rr] .
\end{equation}
Within the current proof (and only here), we change the definition of the measure $\la \cdot \ra$, and set it to be such that, for every bounded measurable function $f : \mcl H_N^D \to \R$,
\begin{multline*}  
\la f(\si) \ra = \exp (-X_K) \int f(\si) \exp \bigg( \sqrt{2t} H_N(\si) - Nt\xi\Ll(\frac{\si \si^*}{N}\Rr) 
\\
+  \sum_{k = 0}^K \Ll( 2 q_k - 2 q_{k-1} \Rr) ^\frac 1 2 y_k \cdot \si - \si\cdot q_K \si \bigg)  \, \d P_N(\si).
\end{multline*}
Although this is implicit in the notation, the measure $\la \cdot \ra$ depends on the choice of the parameters $y_0,\ldots, y_K \in \mcl H_N^D$. (It also depends  on the realization of $(H_N(\si))_{\si \in \mcl H_N^D}$, which is kept fixed throughout the proof and could have been absorbed into the definition of the measure $P_N$.) For every $k \le \ell \in \{0,\ldots, K\}$, we write
\begin{equation*}  
D_{k,\ell} = D_{k\ell} := \frac{\exp \Ll( \zeta_k X_k + \cdots + \zeta_\ell X_\ell \Rr) }{\E_{y_k}\Ll[\exp(\zeta_k X_k)\Rr] \cdots \E_{y_\ell}\Ll[\exp(\zeta_\ell X_\ell)\Rr]}.
\end{equation*}
We decompose the rest of the proof into two steps.

\emph{Step 1.}
We show that, for every $k \in \{0,\ldots, K\}$,
\begin{equation}  
\label{e.expr.drql}
\dr_{q_k} \td F_N = \frac{\zeta_{k+1} - \zeta_k}{N} \E_y \Ll[ \Ll(\E_{y_{\ge k+1}} \Ll[ \la \si \ra D_{k+1,K} \Rr]\Rr) \Ll( \E_{y_{\ge k+1}} \Ll[ \la \si \ra D_{k+1,K} \Rr] \Rr)^* D_{1k}  \Rr] .
\end{equation}
Recall that the expression on the right side above is interpreted according to \eqref{e.def.tausi*}. As in Step 1 of the proof of \cite[Lemma~2.4]{parisi}, one can first show by decreasing induction on $k$ that, for every $k,\ell \in \{0,\ldots, K\}$,
\begin{equation}
\label{e.drq.X}
\dr_{q_\ell} X_{k-1} = \E_{y \ge k} \Ll[ (\dr_{q_\ell} X_k) D_{k K} \Rr] ,
\end{equation}
and similarly, for every $k,\ell \in \{0,\ldots, K\}$ with $k < \ell$,
\begin{equation}  
\label{e.dry.X}
\dr_{y_\ell} X_{k-1} = \E_{y \ge k} \Ll[ (\dr_{y_\ell} X_k) D_{k K} \Rr] ,
\end{equation}
while $\dr_{y_\ell} X_{k-1} = 0$ if $k \ge \ell$. Moreover, for every $\ell \in \{0,\ldots, K-1\}$ and $a \in S^D$, we have
\begin{equation*}  
a \cdot \dr_{q_\ell} X_K = \la D_{\sqrt{2 q_\ell - 2q_{\ell-1}}}(a) y_\ell \cdot \si - D_{\sqrt{2q_{\ell+1} - 2q_\ell}}(a) y_{\ell+1} \cdot \si \ra.
\end{equation*}
Since we are ultimately interested in $\dr_{q_\ell} X_{-1}$, and considering \eqref{e.drq.X}, we need to study
\begin{equation*}  
\E_y  \Ll[\la D_{\sqrt{2 q_\ell - 2q_{\ell-1}}}(a)   y_\ell \cdot\si \ra D_{1K}\Rr] .
\end{equation*}
(Notice that $D_{0K} = D_{1K}$ since $\zeta_0 = 0$.) 
By Gaussian integration by parts, see for instance \cite[(A.1)-(A.2)]{bipartite}, the quantity above can be rewritten as
\begin{multline}  
\label{e.decomp.der}
\E_y \Ll[\la D_{\sqrt{2 q_\ell - 2q_{\ell-1}}}(a) \si \cdot  (2q_\ell - q_{\ell-1})^\frac 1 2 (\si - \si') \ra D_{1K} \Rr] 
\\
+ \E_y  \Ll[\la D_{\sqrt{2 q_\ell - 2q_{\ell-1}}}(a) \si \ra \cdot \dr_{y_\ell} D_{1K}\Rr] ,
\end{multline}
where $\si'$ denotes an independent copy of the random variable $\si$ under $\la \cdot \ra$. Moreover,
\begin{equation*}  
D_{\sqrt{2 q_\ell - 2q_{\ell-1}}}(a) \si \cdot  (2q_\ell - q_{\ell-1})^\frac 1 2 \si = D_{\sqrt{2 q_\ell - 2q_{\ell-1}}}(a)  (2q_\ell - q_{\ell-1})^\frac 1 2 \cdot \si \si^*,
\end{equation*}
and since $\si \si^*$ is a symmetric matrix, we can use \eqref{e.sqrt.dot.sym} to obtain that
\begin{equation*}  
D_{\sqrt{2 q_\ell - 2q_{\ell-1}}}(a) \si \cdot  (2q_\ell - q_{\ell-1})^\frac 1 2 \si = a \cdot \si \si^*. 
\end{equation*}
Similarly, using that the matrix $\la \si \si'^*\ra = \la \si \ra \la \si \ra^*$ is symmetric, we have
\begin{equation*}  
\la D_{\sqrt{2 q_\ell - 2q_{\ell-1}}}(a) \si \cdot  (2q_\ell - q_{\ell-1})^\frac 1 2 \si'\ra = \la a \cdot \si \si'^*\ra. 
\end{equation*}
Combining these observations allows us to identify the first term in \eqref{e.decomp.der} as
\begin{equation*}  
\E_y \Ll[\la a \cdot \si \si^* - a \cdot \si \si'^* \ra D_{1K} \Rr].
\end{equation*}
Turning to the second term in \eqref{e.decomp.der}, we have
\begin{equation*}  
\dr_{y_{\ell}} D_{1K} =  \Ll(\sum_{k = \ell}^K \zeta_k \dr_{y_{\ell}} X_{k} 
     - \sum_{k = \ell+1}^K \zeta_{k} \frac{\E_{y_{k}} \Ll[\dr_{y_{\ell}} X_{k} \exp \Ll( \zeta_k X_k \Rr)   \Rr]}{\E_{y_k} \Ll[ \exp \Ll( \zeta_k X_k \Rr)  \Rr] }\Rr) D_{1K}.
\end{equation*}
Using \eqref{e.dry.X}, we see that, for every $k \ge \ell$,
\begin{equation*}  
\dr_{y_\ell} X_{k} = \E_{y_{\ge k+1}} \Ll[ \la (2q_\ell - 2q_{\ell-1})^\frac 1 2 \si \ra D_{k+1, K} \Rr] ,
\end{equation*}
with the understanding that $\E_{y \ge K+1}$ is the identity map, and $D_{K+1,K} = 1$. 
It thus follows that
\begin{equation*}  
\frac{\E_{y_{k}} \Ll[\dr_{y_{\ell}} X_{k} \exp \Ll( \zeta_k X_k \Rr)   \Rr]}{\E_{y_k} \Ll[ \exp \Ll( \zeta_k X_k \Rr)  \Rr] } = \E_{y \ge k} \Ll[ \la (2q_\ell - 2q_{\ell-1})^\frac 1 2 \si \ra D_{k K} \Rr] ,
\end{equation*}
and
\begin{align*}  
\dr_{y_{\ell}} D_{1K} 
& =  (2q_\ell - 2q_{\ell-1})^\frac 1 2 \Ll(\sum_{k = \ell}^K \zeta_k \E_{y_{\ge k+1}} \Ll[ \la \si \ra D_{k+1,K} \Rr]  - \sum_{k = \ell+1}^K \zeta_{k} \E_{y_{\ge k}} \Ll[ \la \si \ra D_{k K} \Rr] \Rr) D_{1K}
\\
& = (2q_\ell - 2q_{\ell-1})^\frac 1 2 \Ll(\la \si \ra - \sum_{k = \ell}^K  (\zeta_{k+1} - \zeta_k) \E_{y_{\ge k+1}} \Ll[ \la \si \ra D_{k+1,K} \Rr] \Rr) D_{1K}.
\end{align*}
Using again that $\si \si^*$ is symmetric, we deduce that the second term in \eqref{e.decomp.der} can be rewritten as
\begin{equation*}  
\E_y \Ll[ \la a \cdot \si \si^* \ra - \la D_{\sqrt{2q_\ell - 2q_{\ell-1}}}(a) \si \ra \cdot  \sum_{k = \ell}^K  (\zeta_{k+1} - \zeta_k)(2q_\ell - 2q_{\ell-1})^\frac 1 2  \E_{y_{\ge k+1}} \Ll[\la \si \ra D_{k+1,K} \Rr] D_{1K} \Rr] .
\end{equation*}
Using the decomposition $D_{1K} = D_{1k} D_{k+1,K}$ and the fact that $D_{1k}$ does not depend on $y_{k+1},\ldots, y_K$, we see that, for each $k \in \{\ell,\ldots, K\}$,
\begin{align*}  
& \E_y \Ll[D_{\sqrt{2q_\ell - 2q_{\ell-1}}}(a)\la  \si \ra \cdot  (2q_\ell - 2q_{\ell-1})^\frac 1 2  \E_{y_{\ge k+1}} \Ll[ \la \si \ra D_{k+1,K} \Rr] D_{1K} \Rr] 
\\
& \qquad = \E_y  \Ll[D_{\sqrt{2q_\ell - 2q_{\ell-1}}}(a)\E_{y_{\ge k+1}} \Ll[\la  \si \ra D_{k+1,K}  \Rr] \cdot  (2q_\ell - 2q_{\ell-1})^\frac 1 2  \E_{y_{\ge k+1}} \Ll[ \la \si \ra D_{k+1,K} \Rr] D_{1k} \Rr] 
\\
& \qquad = 
\E_y  \Ll[a \cdot \Ll(\E_{y_{\ge k+1}} \Ll[\la  \si \ra D_{k+1,k}  \Rr] \E_{y_{\ge k+1}} \Ll[\la  \si \ra D_{k+1,K}  \Rr]^* \Rr)D_{1k} \Rr] .
\end{align*}
Summarizing, we have shown that, for every $\ell \in \{0,\ldots, K-1\}$,
\begin{equation*}  
- \dr_{q_\ell} X_{-1} =  (\zeta_{\ell+1} - \zeta_\ell)\E_y  \Ll[\Ll(\E_{y_{\ge \ell+1}} \Ll[\la  \si \ra D_{\ell+1,K}  \Rr]\Rr)\Ll( \E_{y_{\ge \ell+1}} \Ll[\la  \si \ra D_{\ell+1,K}  \Rr]\Rr)^* D_{1\ell} \Rr] .
\end{equation*}
By \eqref{e.tdF.is.X-1}, this is \eqref{e.expr.drql}.

\emph{Step 2.} The fact that $\dr_{q_k} \td F_N$ is positive semidefinite is clear from \eqref{e.expr.drql}, as an average of matrices of the form $\tau \tau^*$, $\tau \in \mcl H_N^D$. Showing \eqref{e.monotone.cone} is equivalent to showing that, for every $k \le \ell \in \{0,\ldots,K\}$ and $v \in \R^D$, 
\begin{equation*}  
(\zeta_{k+1} - \zeta_k)^{-1} v\cdot \dr_{q_k} \td F_N v \le (\zeta_{\ell+1} - \zeta_\ell)^{-1} v\cdot \dr_{q_\ell} \td F_N v .
\end{equation*}
Recall that, for any $\tau \in \mcl H_N^D$, we can interpret $\tau^*$ as the mapping in \eqref{e.sigma*}. In view of~\eqref{e.expr.drql}, we need to show that
\begin{equation}
\label{e.monotone.cone.2}
\E_y \Ll[ \Ll| \E_{y_{\ge k+1}} \Ll[ \la \si \ra^* v D_{k+1,K} \Rr]\Rr|^2 D_{1k}  \Rr] \le \E_y \Ll[ \Ll| \E_{y_{\ge \ell+1}} \Ll[ \la \si \ra^* v D_{\ell+1,K} \Rr]\Rr|^2 D_{1\ell}  \Rr] .
\end{equation}
By Jensen's inequality, we have
\begin{align*}  
 \Ll| \E_{y_{\ge k+1}} \Ll[ \la \si \ra^*v D_{k+1,K} \Rr]\Rr|^2 
 & = \Ll|\E_{y_{k+1}, \ldots, y_\ell} \Ll[ \E_{y_{\ge \ell+1}} \Ll[ \la \si \ra^* v D_{\ell+1,K} \Rr] \, D_{k+1,\ell}  \Rr] \Rr|^2 
 \\
& \le \E_{y_{k+1}, \ldots, y_\ell} \Ll[ \Ll|\E_{y_{\ge \ell+1}} \Ll[ \la \si \ra^* v D_{\ell+1,K} \Rr]\Rr|^2 \, D_{k+1,\ell}  \Rr] ,
\end{align*}
and therefore
\begin{align*}  
\E_y \Ll[ \Ll| \E_{y_{\ge k+1}} \Ll[ \la \si \ra^* v D_{k+1,K} \Rr]\Rr|^2 D_{1k}  \Rr] 
& 
\le \E_y \Ll[ \E_{y_{k+1}, \ldots, y_\ell} \Ll[ \Ll|\E_{y_{\ge \ell+1}} \Ll[ \la \si \ra^* v D_{\ell+1,K} \Rr]\Rr|^2 \, D_{k+1,\ell}  \Rr] D_{1k}  \Rr]
\\
& = \E_y \Ll[  \Ll|\E_{y_{\ge \ell+1}} \Ll[ \la \si \ra^* v D_{\ell+1,K} \Rr]\Rr|^2 \, D_{1\ell}  \Rr],
\end{align*}
as desired.
\end{proof}
Up to a simple approximation procedure, the property \eqref{e.monotone} can be rephrased as a monotonicity property for the function $\mu \mapsto \td F_N(t,\mu)$. Recall the definition of $M_\mu$ from Proposition~\ref{p.uparrow.bij}. 
\begin{proposition}[Monotonicity of $\td F_N$]
\label{p.monotone}
The mapping $\mu \mapsto \td F_N(t,\mu)$ is increasing, in the sense that, for every $\mu,\nu \in \mcl P^\upa_1(S^D_+)$,
\begin{equation}  
\label{e.impl.monotone}
M_\mu \le M_\nu  \quad \implies \quad \td F_N(t,\mu) \le \td F_N(t,\nu). 
\end{equation}
\end{proposition}
\begin{proof} In \eqref{e.impl.monotone}, the statement $M_\mu \le M_\nu$ is understood as a pointwise inequality over the interval $[0,1)$. 
Arguing as in Step 2 of the proof of Proposition~\ref{p.lip}, we see that the property in \eqref{e.monotone} allows us to show the implication \eqref{e.impl.monotone} for any pair of measures $\mu,\nu \in \mcl P^\uparrow(S^D_+)$ of finite support. We now consider general $\mu,\nu \in \mcl P^\upa_1(S^D_+)$, and argue by approximation. 
For each integer $K \ge 1$, the mapping 
\begin{equation*}  
\Ll\{
\begin{array}{rcl}  
[0,1) & \to & S^D_+ \\
u & \mapsto & M_\mu \Ll( \frac{\lfloor K u \rfloor}{K} \Rr) 
\end{array}
\Rr.
\end{equation*}
is right-continuous with left limits, and is increasing. We denote by $\mu^{(K)} \in \mcl P^\upa(S^D_+)$ the measure such that the mapping above is $M_{\mu^{(K)}}$. In other words, $\mu^{(K)}$ is the law of the random variable $M_\mu \Ll( \frac{\lfloor K U \rfloor}{K} \Rr)$.
We construct the measure $\nu^{(K)}$ similarly, replacing $\mu$ by $\nu$ throughout. Under the assumption that $M_\mu \le M_\nu$, we have $M_{\mu^{(K)}} \le N_{\nu^{(K)}}$. Given that the proposition is valid for measures of finite support, it thus suffices to show that $\td F_N(t,\mu^{(K)})$ converges to $\td F_N(t,\mu)$ as $K$ tends to infinity (which implies the same result for~$\nu$). In view of Proposition~\ref{p.lip}, it suffices to show that 
\begin{equation*}  
\lim_{K \to \infty} \E \Ll[ |X_{\mu^{(K)}} - X_{\mu} |\Rr] = 0.
\end{equation*}
This follows from the dominated convergence theorem. 
\end{proof}
In a similar fashion, we can use an approximation argument and combine the two properties appearing in Proposition~\ref{p.basic.ineq} into the following statement.
\begin{proposition}
\label{p.tilted}
For every $\mu,\nu \in \mcl P^\upa_1(S^D_+)$, we have
\begin{equation}  
\label{e.tilted}
\Ll( \forall u \in [0,1], \ \int_u^1 M_\mu(r) \, \d r \le  \int_u^1 M_\nu(r) \, \d r  \Rr) \quad  \implies \quad  \td F_N(t,\mu) \le \td F_N(t,\nu).  
\end{equation}
\end{proposition}
\begin{proof}
We first show that the statement is valid for measures $\mu,\nu \in \mcl P^\upa(S^D_+)$ that can be written, for some integer $K \ge 1$, in the form
\begin{equation}  
\label{e.k.diracs}
\mu = \frac 1 K \sum_{k = 1}^K \de_{q_k}, \qquad \nu = \frac 1 K \sum_{k = 1}^K \de_{q'_k},
\end{equation}
with parameters $q_1, \ldots, q_K, q'_1, \ldots, q'_K \in S^D_+$ satisfying
\begin{equation*}  
q_1 \le \cdots \le q_K, \qquad q'_1 \le \cdots \le q'_K.
\end{equation*}
In this case, the property on the left side of \eqref{e.tilted} translates into
\begin{equation}  
\label{e.ass.order.q}
\forall k \in \{1,\ldots, K\}, \ \sum_{\ell = k}^K q_\ell \le \sum_{\ell = k}^K q'_\ell.
\end{equation}
For every $s \in [0,1]$, we write $\mu_s := \sum_{k = 1}^K \de_{(1-s) q_k + s q_k'}$, and observe that, by discrete integration by parts,
\begin{align*}  
\td F_N(t,\nu) - \td F_N(t,\mu) 
& 
= \sum_{k = 1}^K \int_0^1 (q_k' - q_k) \cdot \dr_{q_k} \td F_N(t,\mu_s) \, \d s
\\
& 
= \int_0^1 \sum_{k = 1}^K \sum_{\ell = k}^K (q_\ell' - q_\ell) \cdot \Ll( \dr_{q_k} \td F_N(t,\mu_s) - \dr_{q_{k-1}} \td F_N(t,\mu_s) \Rr) \, \d s,
\end{align*}
with the understanding that $\dr_{q_{0}} \td F_N = 0$. (For notational convenience, the indexing of the support of the measures starts at $k = 1$ here, unlike in the rest of this section, but similarly to the next sections.) By Proposition~\ref{p.monotone}, we have
\begin{equation*}  
\dr_{q_k} \td F_N(t,\mu_s) - \dr_{q_{k-1}} \td F_N(t,\mu_s) \ge 0.
\end{equation*}
(Notice that we use \eqref{e.monotone} in the case $k = 1$.) 
By \eqref{e.ass.order.q}, we also have that 
\begin{equation*}  
\sum_{\ell = k}^K (q_\ell' - q_\ell) \ge 0.
\end{equation*}
Since the dot product of two matrices in $S^D_+$ is nonnegative ($a \cdot b = |\sqrt a \sqrt b|^2$), we obtain the result for measures $\mu,\nu$ of the form \eqref{e.k.diracs}. To conclude, we can then argue as in the proof of Proposition~\ref{p.monotone}. 
\end{proof}

%
%
%
%
%
%

\section{Viscosity solutions}
\label{s.visc}

The goal of this section is to give a rigorous meaning to the partial differential equation in \eqref{e.hj}. The approach taken up here is to define the solution as the limit of finite-dimensional approximations. Since this specific aspect was covered in rather wide generality in \cite{bipartite}, we will be able to borrow several ingredients from there. Compared with \cite{bipartite}, there are however important differences that relate to the handling of the boundary condition. As will be explained below, these difficulties come from the fact that the geometry of the set of positive definite matrices is more intricate than that of that of the set $\R^2_+$ (or $\R^D_+$) that replaces it in \cite{bipartite}. We will bypass these difficulties by modifying the nonlinearity in the equation, outside of its ``natural'' domain of definition; see in particular \eqref{e.extend.H} below. 

\subsection{Analysis of finite-dimensional equations}
Let $K \ge 1$ be an integer. We define the open set
\begin{equation}
\label{e.def.Uk}
U_K := \Ll\{ x  = (x_1,\ldots, x_K) \in (S_{++}^D)^K \ : \ \forall k \in \{1,\ldots, K-1\}, \ x_{k+1} - x_k \in S^D_{++}  \Rr\},
\end{equation}
and its closure
\begin{equation}
\label{e.def.barUk}
\bar U_K := \Ll\{ x  = (x_1,\ldots, x_K) \in (S_+^D)^K \ : \ x_1 \le \cdots \le x_K  \Rr\}.
\end{equation}
The finite-dimensional equations we consider take the form
\begin{equation}
\label{e.hj.finite.nobd}
\dr_t f - \msf H(\nabla f) = 0 \qquad \text{in } (0,T) \times U_K,
\end{equation}
for a given locally Lipschitz function $\msf H : (S^D)^K \to \R$, $T \in (0,\infty]$, and with a prescribed initial condition at $t = 0$. In the expression above, we use the notation, with the understanding that $f = f(t,x)$ with $x = (x_1,\ldots, x_K)$, 
\begin{equation*}  
\nabla f := (\dr_{x_1} f, \ldots, \dr_{x_K} f).
\end{equation*}
Recall that each $\dr_{x_k} f$ takes values in the set $S^D$. We also impose a Neumann boundary condition on $\dr U_K$ for solutions to \eqref{e.hj.finite.nobd}. We define the outer normal to a point $x \in \dr U_K$ as the set
\begin{equation*}  
\n(x) := \Ll\{ \nu \in (S^D)^K \ : |\nu| = 1 \ \text{ and } \ \forall y \in \bar U_K, \ (y-x) \cdot \nu  \le 0\Rr\} .
\end{equation*}
To display the Neumann boundary condition, we write the equation formally as
\begin{equation}
\label{e.hj.finite}
\Ll\{
\begin{aligned}  
& \dr_t f - \msf H(\nabla f) = 0 & \qquad & \text{in } (0,T) \times U_K, \\
& \n \cdot \nabla f = 0 & \qquad & \text{on } (0,T)\times \dr U_K.
\end{aligned}
\Rr.
\end{equation}

\begin{definition}  
\label{def.solution}
We say that a function $f \in C([0,T) \times \bar U_K)$ is a \emph{viscosity subsolution} to~\eqref{e.hj.finite} if for every $(t,x) \in (0,T)\times \bar U_K$ and $\phi \in C^\infty((0,T)\times  \bar U_K)$ such that $(t,x)$ is a local maximum of $f-\phi$, we have
\begin{equation}  
\label{e.interior.cond.subsol}
\Ll(\dr_t \phi - \H(\nabla \phi)\Rr)(t,x) \le 0 \quad \text{ if } x \in U_K,
\end{equation}
and
\begin{equation}  
\label{e.bdy.cond.subsol}
\min \Ll( \inf_{\nu \in \n(x)} \nabla \phi \cdot \nu, \dr_t \phi - \H(\nabla \phi)\Rr) (t,x) \le 0 \quad \text{ if } x \in \dr U_K.
\end{equation}
We say that a function $f \in C([0,T) \times \bar U_K)$ is a \emph{viscosity supersolution} to~\eqref{e.hj.finite} if for every $(t,x) \in (0,T)\times \bar U_K$ and $\phi \in C^\infty((0,T)\times \bar U_K)$ such that $(t,x)$ is a local minimum of $f-\phi$, we have 
\begin{equation*}  
\Ll(\dr_t \phi - \H(\nabla \phi)\Rr)(t,x) \ge 0 \quad \text{ if } x \in U_K,
\end{equation*}
and
\begin{equation}  
\label{e.bdy.cond.supersol}
\max \Ll( \sup_{\nu \in \n(x)} \nabla \phi \cdot \nu, \dr_t \phi - \H(\nabla \phi)\Rr) (t,x) \ge 0  \quad \text{ if } x \in \dr U_K.
\end{equation}
We say that a function $f \in C([0,T) \times \bar U_K)$ is a \emph{viscosity solution} to~\eqref{e.hj.finite} if it is both a viscosity subsolution and a viscosity supersolution to \eqref{e.hj.finite}. 
\end{definition}

We often drop the qualifier \emph{viscosity} and simply talk about subsolutions, supersolutions, and solutions to~\eqref{e.hj.finite}. We say that a function $f \in C([0,T) \times \bar U_K)$ is a \emph{solution} to
\begin{equation}
\label{e.hj.finite.subsol}
\Ll\{
\begin{aligned}  
& \dr_t f - \msf H(\nabla f) \le 0 & \qquad & \text{in } (0,T) \times U_K, \\
& \n \cdot \nabla f \le 0 & \qquad & \text{on } (0,T)\times \dr U_K,
\end{aligned}
\Rr.
\end{equation}
whenever it is a subsolution to \eqref{e.hj.finite}; and similarly with the inequalities reversed for supersolutions.

We now recall the comparison principle proved in \cite[Proposition~3.2]{bipartite}, which in particular implies, for a given initial condition, the uniqueness of solutions.

\begin{proposition}[Comparison principle]
\label{p.comp}
Let $T \in (0,\infty]$, and let $u$ and $v$ be respectively a sub- and a super-solution to \eqref{e.hj.finite} that are both uniformly Lipschitz continuous in the $x$ variable. 
We have
\begin{equation}  
\label{e.comp}
\sup_{\Ll[0,T\Rr) \times \bar U_K} (u-v) = \sup_{\{0\}\times \bar U_K} (u-v).
\end{equation}
\end{proposition}
For any $x = (x_1,\ldots, x_K)\in (S^D)^K$, we use the notation 
\begin{equation*}  
|x| := \Ll( \sum_{k = 1}^K |x_k|^2\Rr) ^\frac 1 2.
\end{equation*}
The existence and regularity of solutions are provided by the next result, borrowed from \cite[Proposition~3.4]{bipartite}. 

\begin{proposition}[Existence of solutions]
\label{p.existence}
For every uniformly Lipschitz initial condition $f_0 : \bar U_K \to \R$, there exists a viscosity solution $f$ to \eqref{e.hj.finite} that satisfies $f(0,\cdot) = f_0$. Moreover, the function $f$ is Lipschitz continuous, and we have 
\begin{equation}  
\label{e.lip.hj.finite}
\| \, |\nabla f| \, \|_{L^\infty(\R_+\times U_K)} = \| \, |\nabla f_0| \, \|_{L^\infty(U_K)} .
\end{equation}
\end{proposition}

\subsection{Tilted functions and boundary condition}
Our strategy for the verification of the boundary condition relies crucially on the monotonicity properties of the free energy, see Propositions~\ref{p.basic.ineq} and \ref{p.tilted}. In order to discuss this precisely, it is convenient to introduce the cone dual to $\bar U_K$, denoted $\bar U_K^*$. We recall the following result from \cite[Lemma~3.3]{bipartite}.
\begin{lemma}[Dual cone to $\bar U_K$]
\label{l.dual.cone}
Let $\bar U_K^*$ denote the cone dual to $\bar U_K$, that is,
\begin{equation}  
\label{e.def.baruk*}
\bar U_K^* := \Ll\{ x \in (S^D)^K \ : \ \forall v\in \bar U_K, \ x \cdot v \ge 0\Rr\} .
\end{equation}
We have
\begin{equation}  
\label{e.ident.baruk*}
\bar U_K^* = \Ll\{ x \in (S^D)^K \ : \ \forall k \in \{ 1,\ldots, K\}, \ \sum_{\ell = k}^K x_\ell \ge 0 \Rr\} ,
\end{equation}
and
\begin{equation}  
\label{e.double.dual.uk}
\bar U_K = \Ll\{ v \in (S^D)^K \ : \ \forall x \in \bar U_K^*, \ x \cdot v \ge 0 \Rr\} .
\end{equation}
\end{lemma}
For $V$ a subset of $(S^D)^K$, we say that a function $f : V \to \R$ is \emph{tilted} if, for every $x,y \in V$, we have
\begin{equation*}  
y - x \in \bar U_K^*\quad \implies \quad f(x) \le f(y),
\end{equation*}
For any interval $I \subset \R_+$, we say that a function $f : I \times V \to \R$ is \emph{tilted} if the function $f(t,\cdot)$ is tilted for every $t \in I$. We recall from \cite[Lemma~3.5]{bipartite} that a Lipschitz function $f$ is tilted if and only if $\nabla f \in \bar U_K$ almost everywhere. 
By Proposition~\ref{p.tilted}, we see that the mapping
\begin{equation}  
\label{e.barFN.on.uk}
\Ll\{
\begin{array}{rcl}  
\R_+ \times \bar U_K & \to & \R \\
(t,x) & \mapsto & \bar F_N \Ll( t, \frac 1 K \sum_{k = 1}^K \de_{x_k} \Rr) 
\end{array}
\Rr.
\end{equation}
is tilted; and therefore,  by \cite[Proposition~3.6]{bipartite}, that it
satisfies the boundary condition for being a subsolution to \eqref{e.hj.finite}.

In the remainder of this subsection, we explain why the strategy used in \cite{bipartite} for the verification of the boundary condition for being a supersolution cannot be applied in the present more general setting. As just discussed, the gradient $\nabla \bar F_N$ of the mapping in \eqref{e.barFN.on.uk} belongs to $\bar U_K$. Suppose that $\phi$ is a smooth function such that $\bar F_N - \phi$ has a local minimum at $(t,x)$. If $x \in U_K$, it then follows that $\nabla \phi  = \nabla \bar F_N$, and in particular $\nabla \phi \in \bar U_K$. Notice also that if $x = 0$, then the cone generated by $\n(x) = \n(0)$ is $-\bar U_K^*$, by~\eqref{e.def.baruk*}. Using also \eqref{e.double.dual.uk}, we thus see that
\begin{equation*}  
\sup_{\nu \in \n(0)} \nu \cdot \nabla \phi(t,0) \ge 0 \quad \text{ or } \quad \nabla \phi(t,0) \in \bar U_K. 
\end{equation*}
Whenever $x \in U_K$ or $x = 0$, we thus only have to verify the inequality
\begin{equation*}  
(\dr_t \phi - \H(\nabla \phi))(t,x) \ge 0
\end{equation*}
in situations for which $\nabla \phi(t,x) \in \bar U_K$. In \cite{bipartite} (see Step 5 of the proof of Theorem~4.1 there), we could extend this observation to every $x \in \bar U_K$. However, this relied on particular properties of the simpler geometry of the domain under consideration there, in which $S^D_+$ is replaced by $\R_+^2$ (or $\R_+^D$). This crucial point is no longer valid in our context, as we explain now. That is, we can find a tilted function $f$, a smooth test function $\phi$, and a contact point $(t,x) \in (0,\infty)\times \bar U_K$ such that $f-\phi$ has a local minimum at $(t,x)$, and yet
\begin{equation*}  
\sup_{\nu \in \n(x)} \nu \cdot \nabla \phi(t,x) < 0  \quad \text{ and } \quad \nabla \phi(t,x) \notin \bar U_K. 
\end{equation*}
We give an example for which $f$ and $\phi$ are constant in time and linear in $x$ (one can then add a linear function of time to $f$ if one wants to ensure that $f$ is indeed a supersolution). In this case, the condition that $(t,x)$ be a local minimum of $f-\phi$ is equivalent to
\begin{equation}  
\label{e.boundary.notriv}
\forall y \in \bar U_K, \quad (y-x) \cdot \nabla (f - \phi)  \ge 0.
\end{equation}
We consider, for $K = 1$ and $D = 2$, the choice of
\begin{equation*}  
x = 
\begin{pmatrix}  
0 & 0 \\
0 & 1
\end{pmatrix},
\qquad 
\nabla f = 
\begin{pmatrix}  
4 & 2 \\
2 & 1 
\end{pmatrix},
\quad \text{ and } \quad
\nabla \phi = 
\begin{pmatrix}  
1 & 2 \\
2 & 1
\end{pmatrix}.
\end{equation*}
We clearly have $x \in S^2_+ = \bar U_1$, $\nabla f \in S^2_+ = \bar U_1$ (and therefore $f$ is tilted), and $\nabla \phi \notin S^2_+ = \bar U_1$. The property in \eqref{e.boundary.notriv} is satisfied since
\begin{equation*}  
\nabla (f-\phi) = 
\begin{pmatrix}  
3 & 0 \\
0 & 0
\end{pmatrix}
\end{equation*}
and for every $y \in S^2_+$, we have that the $(1,1)$ entry of the matrix $y-x$ is that of $y$, which is nonnegative, since $y \in S^2_+$. On the other hand, since here $K = 1$, we have that
\begin{equation*}  
\n(x) = \Ll\{ \nu \in S^D \ : \ |\nu| = 1 \ \text{ and } \ \forall y \in S^2_+, \ (y-x) \cdot \nu \le 0 \Rr\} .
\end{equation*}
Since we can in particular select $y \in x + S^2_+$ in this definition, we see using \eqref{e.charact.psd} that every $\nu \in \n(x)$ must be such that $-\nu \in S^2_+$. Choosing diagonal matrices for $y$, we also see that the $(2,2)$ entry of $\nu$ must be zero. Since $-\nu \in S^2_+$, the Cauchy-Schwarz inequality then yields that $-\nu$ must be diagonal; and since it is normalized to be of unit norm, we conclude that
\begin{equation*}  
\n(x) = \Ll\{
\begin{pmatrix}  
-1 & 0 \\
0 & 0
\end{pmatrix}
\Rr\}.
\end{equation*}
In particular, the condition $\nabla \phi \cdot \nu < 0$ is indeed satisfied for every $\nu \in \n(x)$. 

The situation described here is unlike the one faced in \cite{bipartite}. In this earlier work, in some sense the problem ``only looks at the diagonal elements of the matrices''; for instance, the ``counterexample'' above simplifies into $x = (0,1)$, $\nabla f = (4,1)$, $\nabla \phi = (1,1)$, and so $\nabla \phi \in \R_+^2$, as desired. 
Our strategy to circumvent this difficulty is to find a suitable modification of the nonlinearity outside of $\bar U_K$, see \eqref{e.extend.H} below.

\subsection{Convergence of finite-dimensional approximations}
We now proceed to state the convergence of finite-dimensional approximations to \eqref{e.hj}. In order to do so, we first need to find good ``representatives'' of a measure $\mu \in \mcl P^\upa(S^D)$ within the sets $\bar U_K$. Mapping an element of $\bar U_K$ into a measure is straightforward: for every $x \in \bar U_K$, we can simply consider the measure
\begin{equation}  
\label{e.map.measure}
\frac 1 K \sum_{k = 1}^K \de_{x_k}.
\end{equation}
There is some flexibility to define a converse operation, mapping a given measure to an element of $\bar U_K$. Fixing $K \ge 1$, and recalling the definition of $M_\mu$ from Proposition~\ref{p.uparrow.bij}, we set, for every $\mu \in \mcl P^\upa_1(S^D)$ and $k \in \{1,\ldots, K\}$,
\begin{equation}
\label{e.def.mu.to.xk}
x^{(K)}_k(\mu) := K \int_{\frac{k-1}{K}}^{\frac k K} M_{\mu}(u) \, \d u.
\end{equation}
Notice that $x^{(K)}(\mu) := (x^{(K)}_k(\mu))_{1 \le k \le K}$ belongs to $\bar U_K$. 

We next discuss the discretization of the nonlinearity in \eqref{e.hj}. Recalling \eqref{e.def.dmu}, and for $\mu$ of the form in \eqref{e.map.measure}, we have
\begin{equation*}  
\int \bar \xi(\dr_\mu f) \, \d \mu = \frac 1 K \sum_{k = 1}^K \bar \xi(K\dr_{x_k} f).
\end{equation*}
Recall also that, for the function of interest to us, namely $\bar F_N$ (or more precisely: the function appearing in \eqref{e.barFN.on.uk}) we have that $\nabla \bar F_N \in \bar U_K$; see the discussion around \eqref{e.barFN.on.uk}. In view of this, we set
\begin{equation}  
\label{e.def.Hk}
\forall p \in \bar U_K, \qquad \H_K(p) := \frac 1 K \sum_{k = 1}^K \bar \xi(K p_k). 
\end{equation}
For almost every $p \in U_K$, we can compute
\begin{equation*}  
\nabla \H_K(p) = \Ll( \nabla \bar \xi(K p_1), \ldots, \nabla \bar \xi(K p_K) \Rr) .
\end{equation*}
Since $\bar \xi$ is proper, we see that $\nabla H_K$ maps $\bar U_K$ to itself. In other words, the mapping $\H_K$ is tilted on $\bar U_K$. 

As discussed in the previous subsection, we need to identify a suitable extension of $\H_K$ to $p \notin \bar U_K$. For every $p \in (S^D)^K$, we set
\begin{equation}  
\label{e.extend.H}
\H_K(p) := \inf \Ll\{ \H_K(q) \ : \ q \in \bar U_K \cap (p+\bar U_K^*) \Rr\} .
\end{equation}
Since $\H_K$ is tilted over $\bar U_K$, whenever $p \in \bar U_K$ and $q \in \bar U_K \cap (p + \bar U_K^*)$, we have that $\H_K(p) \le \H_K(q)$, and therefore the identity \eqref{e.extend.H} is indeed valid when $p \in \bar U_K$. We will see in the course of the proof of Proposition~\ref{p.conv.finite.dim} that $\H_K$ is well-defined, finite, and uniformly Lipschitz over $(S^D)^K$.
Recalling the notation $X_\mu$ from \eqref{e.def.Xmu}, we are now ready to state the main result of this section.
\begin{proposition}[Convergence of finite-dimensional approximations]
\label{p.conv.finite.dim}
Let $\bar{\xi} : S^D_+ \to \R$ be a proper and uniformly Lipschitz function,  $L < \infty$, and $\psi : \mcl P^\upa_1(S^D) \to \R$ be such that for every $\mu,\nu \in \mcl P^\upa_1(S^D)$, 
\begin{equation}  
\label{e.lip.assumption}
\Ll|\psi(\mu) - \psi(\nu)\Rr| \le L \, \E \Ll[ |X_\mu - X_\nu | \Rr] .
\end{equation}
For every integer $K \ge 1$, with $\H_K$ defined in \eqref{e.def.Hk}-\eqref{e.extend.H}, let $f^{(K)} : \R_+ \times \bar U_K \to \R$ be the viscosity solution to
\begin{equation}
\label{e.def.finite.dim}
\Ll\{
\begin{aligned}  
& \dr_t f^{(K)} - \H_K\Ll( \nabla f^{(K)} \Rr) = 0 & \quad \text{in } (0,\infty) \times U_K,
\\
& \n \cdot \nabla f^{(K)} = 0 & \quad \text{on } (0,\infty) \times \dr U_K,
\end{aligned}
\Rr.
\end{equation}
with initial condition given, for every $x \in \bar U_K$, by
\begin{equation}
\label{e.def.finite.dim.init}
f^{(K)}(0,x) = \psi \Ll( \frac 1 K \sum_{k = 1}^K \de_{x_k} \Rr) .
\end{equation}
For every $t \ge 0$ and $\mu \in \mcl P^\upa_2(S^D)$, the following limit exists and is finite:
\begin{equation}  
\label{e.conv.finite.dim}
f(t,\mu) := \lim_{K \to \infty} f^{(K)}\Ll(t,x^{(K)}(\mu)\Rr),
\end{equation}
where on the right side, we use the notation defined in~\eqref{e.def.mu.to.xk}. By definition, we interpret this limit as the solution to \eqref{e.hj}.
Moreover, there exists a constant $C < \infty$ such that for every integer $K \ge 1$, $t\ge 0$, and $\mu, \nu \in \mcl P_2^\upa(S^D)$, we have
\begin{equation}
\label{e.quant.conv}
\Ll| f(t,\mu) - f^{(K)}(t,x^{(K)}(\mu)) \Rr| \le \frac{C}{\sqrt{K}} \Ll( t + \Ll(\E \Ll[ |X_\mu|^2 \Rr]\Rr) ^\frac 1 2  \Rr) ,
\end{equation}
as well as
\begin{equation}  
\label{e.lip.f}
\Ll| f(t,\mu) - f(t,\nu) \Rr| \le L \Ll(\E \Ll[ |X_\mu - X_\nu|^2 \Rr] \Rr)^\frac 1 2.
\end{equation}
\end{proposition}
\begin{remark}  
\label{r.notation.tricky}
Our redefinition of the nonlinearities $\H_K$ outside of $\bar U_K$, see \eqref{e.extend.H}, renders the notation in \eqref{e.hj} somewhat misleading. Indeed, since the identity in \eqref{e.def.Hk} is only valid for $p \in \bar U_K$, the notation in \eqref{e.hj} is only really legitimate if the solution to \eqref{e.def.finite.dim} satisfies $\nabla f^{(K)} \in \bar U_K$. For tilted initial conditions $\psi$, this is rather plausible, but this property is \emph{not} proved in the present paper. Recall however that we do know that the gradient of the free energy in \eqref{e.barFN.on.uk} belongs to $\bar U_K$, and we ultimately aim to show that $f^{(K)}$ becomes a sharp approximation of this free energy.
\end{remark}
\begin{proof}[Proof of Proposition~\ref{p.conv.finite.dim}]
The strategy is similar to that for the proof of \cite[Proposition~3.7]{bipartite}, so we only discuss the new ingredients. For all integers $K, R \ge 1$, we set $K' := KR$, and for every $p \in (S^D)^K$, we define
\begin{equation}  
\label{e.def.p'}
p' := \frac 1 R (p_1,\ldots, p_1, p_2,\ldots, p_2, \ldots, p_K, \ldots ,p_K) \in (S^D)^{K'},
\end{equation}
where each term $p_k$ is repeated $R$ times. We claim that
\begin{equation}
\label{e.consistency.H}
\H_{K'}(p') = \H_K(p).
\end{equation}
When $p \in \bar U_K$, this is clear, by \eqref{e.def.Hk}. In general, using this observation and \eqref{e.extend.H}, we readily see that $\H_{K'}(p') \le \H_K(p)$. Let $q \in \bar U_{K'} \cap (p' + \bar U_{K'}^*)$, 
\begin{equation*}  
\td q := \Ll(\frac 1 R \sum_{r = 1}^R q_r, \frac 1 R \sum_{r = 1}^R q_{R+r}, \ldots, \frac 1 R \sum_{r = 1}^R q_{(K-1)R+r} \Rr) \in (S^D)^K,
\end{equation*}
and let $\td q' \in (S^D)^{K'}$ be defined from $\td q$ as $p'$ was defined from $p$ in \eqref{e.def.p'}, repeating each coordinate $R$ times. 
One can verify that $\td q' \in \bar U_{K'} \cap (p' + \bar U_{K'}^*)$. We also have $\td q - \td q' \in \bar U_{K'}^*$, and thus
\begin{equation*}  
\H_{K'}(q) \le \H_{K'}(\td q') = \H_{K}(\td q),
\end{equation*}
where we used \eqref{e.consistency.H} in $\bar U_K$ in the last step. This shows the converse bound $\H_{K'}(p') \ge \H_{K}(p)$, and completes the proof of \eqref{e.consistency.H} for arbitrary $p \in (S^D)^K$. This ensures that the argument in the first step of the proof of \cite[Proposition~3.7]{bipartite} still applies.

In order to get the remainder of this proof to work, the key point is to verify that the nonlinearity $\H_K$ is Lipschitz continuous, uniformly over $K$, once the norms are scaled properly in terms of $K$. The rescaled norm is defined, for every $p \in (S^D)^K$, by
\begin{equation*}  
|p|_{2*} := \Ll( \frac 1 K \sum_{k = 1}^K (K |p_k|)^2 \Rr)^\frac 1 2.
\end{equation*}
We show that there exists a constant $C < \infty$ such that for every $K \ge 1$ and $p,p' \in (S^D)^K$,
\begin{equation}
\label{e.lip.H}
\Ll| \H_K(p') - \H_K(p) \Rr| \le C |p'-p|_{2*}.
\end{equation}
When $p$ and $p'$ both belong to $\bar U_K$, the validity of \eqref{e.lip.H} follows from \eqref{e.def.Hk} and the fact that $\bar \xi$ is Lipschitz continuous.
For the general case, with arbitrary $p,p' \in (S^D)^K$, let $m$ denote the projection of $p'-p$ onto $\bar U_K$. In other words, $m \in \bar U_K$ is the minimizer of the mapping $m' \mapsto |p'-p-m|^2$ over $\bar U_K$.  Since $\bar U_K$ is closed, the point $m$ is well-defined, and since it is also convex, we have
\begin{equation*}  
\forall v \in \bar U_K, \ (v-m) \cdot (p'-p-m) \le 0. 
\end{equation*}
Since $\bar U_K$ is a cone, we can in particular choose $v = \alpha m$ in the expression above, for every $\al \in [0,\infty)$. It follows that
\begin{equation}  
\label{e.orthogonality}
(p'-p-m) \cdot m = 0,
\end{equation}
and thus
\begin{equation*}  
\forall v \in \bar U_K, \ v \cdot (p'-p-m) \le 0. 
\end{equation*}
By \eqref{e.def.baruk*}, this means that $p-p'+m \in \bar U_K^*$. From this, we now deduce that 
\begin{equation}  
\label{e.move.around}
\forall q \in \bar U_K \cap (p + \bar U_K^*), \quad q + m \in \bar U_K \cap(p' + \bar U_K^*). 
\end{equation}
Indeed, since $\bar U_K$ is a convex cone, and $q, m \in \bar U_K$, we clearly have $q+m \in \bar U_K$. By assumption, we also have that $q-p \in \bar U_K^*$, and since $p-p'+m \in \bar U_K^*$, we obtain that $q+m \in p' + \bar U_K^*$. In particular, using \eqref{e.move.around} with $p = 0$, we see that $\bar U_K \cap (p' + \bar U_K^*)$ is not empty. 
Recalling \eqref{e.orthogonality}, we also have that 
\begin{equation*}  
|p'-p|^2 = |p'-p-m|^2 + |m|^2,
\end{equation*}
and in particular, 
\begin{equation*}  
|m| \le |p'-p|.
\end{equation*}
Since $q$ and $q+m$ both belong to $\bar U_K$, we can thus combine this with \eqref{e.lip.H} to get that 
\begin{equation*}  
|\H_K(q+m) - \H_K(q)| \le C  |p'-p|_{2*}.
\end{equation*}
Since this derivation is valid for every $q \in \bar U_K \cap (p + \bar U_K^*)$, we deduce that 
\begin{equation*}  
\H_K(p) \ge \H_K(p') - C |p'-p|_{2*}.
\end{equation*}
Having established this inequality for every $p,p' \in (S^D)^K$, we have thus completed the proof of \eqref{e.lip.H}. We can then follow the arguments from the proof of \cite[Proposition~3.7]{bipartite} to obtain the announced results.
\end{proof}

\section{The free energy is a supersolution}
\label{s.super}

In this section, we show that finite-dimensional approximations of the function $\bar F_N$ are approximate supersolutions to the finite-dimensional equations appearing in \eqref{e.def.finite.dim}. We then combine this result with Proposition~\ref{p.conv.finite.dim} to prove Theorems~\ref{t.main} and \ref{t.main1}. Throughout the section, we fix a regularization $\bar \xi$ of the function $\xi$ appearing in \eqref{e.def.cov}. Recall that the definition of a regularization appears just before the statement of Theorem~\ref{t.main1}, and that the nonlinearity $\H_K$ is defined in \eqref{e.def.Hk}-\eqref{e.extend.H}, and that the function $\bar F_N$ is defined in \eqref{e.def.FN}-\eqref{e.def.barFN}. 

\begin{proposition}[approximate supersolution]
\label{p.super}
There exists a constant $C < \infty$ (depending only on $D$ and $\xi$) such that the following holds. 
Let $K \ge 1$ be an integer, and for every $t \ge 0$ and $q \in \bar U_K$, denote
\begin{equation}  
\label{e.def.FNK}
\bar F_N^{(K)}(t,q) := \bar F_N \Ll( t, \frac 1 K \sum_{k = 1}^K \de_{q_k} \Rr) .
\end{equation}
Let $f$ be any subsequential limit of $\bar F_N^{(K)}$ as $N$ tends to infinity. The function $f$ is a  solution to
\begin{equation}
\label{e.super}
\Ll\{
\begin{aligned}  
& \dr_t f - \H_K(\nabla f ) \ge -\frac{C}{K} & \quad \text{in } (0,\infty) \times U_K,
\\
& \n \cdot \nabla f \ge 0 & \quad \text{on } (0,\infty) \times \dr U_K.
\end{aligned}
\Rr.
\end{equation}
\end{proposition}
Since $\bar F_N(0,\de_0) = 0$ and $\bar F_N$ is Lipschitz continuous, see Proposition~\ref{p.lip} and \eqref{e.expr.drt}, it is clear that the family of functions $(\bar F_N^{(K)})_{N \ge 1}$ is precompact for the topology of uniform convergence. We understand the notion of subsequential limits in the statement of Proposition~\ref{p.super} as referring to this topology. 

As should be apparent from \eqref{e.almost.equation}, the proof of Proposition~\ref{p.super} will rely on the fact that the overlap matrix $\si \si'^*$ is essentially determined by the knowledge of the overlap $\al \wedge \al'$. This ``synchronization'' of the overlaps can be obtained by using the technique introduced in \cite{pan.multi}, which itself is based on the proof of ultrametricity obtained in~\cite{pan.aom}; we will appeal to the finitary version of these results developed in \cite{bipartite} for this purpose. Compared to the setting explored in \cite{bipartite}, a new difficulty arises, since it is a priori only possible to obtain an approximate synchronization of the symmetric part of the matrix~$\si \si'^*$. To overcome this, we will borrow an argument from \cite{pan.potts} allowing to show that the antisymmetric part of the matrix $N^{-1} \si \si'^*$ tends to zero as $N$ tends to infinity. 

As discussed below \eqref{e.almost.equation}, we stress that the synchronization and ``symmetrization'' properties described above will not be shown to hold for every possible choice of the parameters. First, we will add various additional terms to the energy function. Contrary to the ``enrichment'' of the energy function that was performed in Section~\ref{s.enriched}, here the additional terms will be perturbative, in the sense that they will not affect the value of the limit free energy. On the other hand, recall that our goal here is to show that $\bar F_N^{(K)}$ is an approximate supersolution to \eqref{e.super}. By slightly tilting the test function $\phi$ appearing in Definition~\ref{def.solution}, we will be able to ``activate'' these perturbative parameters, and show that the synchronization and symmetrization properties hold \emph{at the contact point} appearing in Definition~\ref{def.solution}, up to a small error. In fact, we will show that at the contact point, the perturbative terms ensure the validity (up to a small error) of certain distributional identities involving multiple overlaps, which are usually called Ghirlanda-Guerra identities; and these in turn imply the sought-after synchronization and symmetrization properties. 

We start by introducing some notation. Let $(\si^\ell, \al^\ell)_{\ell \ge 1}$ be a family of independent copies of $(\si,\al)$ under the Gibbs measure $\la \cdot \ra$. Recall that this measure depends on the choice of parameters $(t,\mu) \in \R_+ \times \mcl P_1^\upa(S^D_+)$; throughout this section we understand that the measure $\mu$ is of the form
\begin{equation*}  
\frac 1 K \sum_{k = 1}^K \de_{q_k},
\end{equation*}
for some integer $K \ge 1$ and $q \in \bar U_K$, in agreement with \eqref{e.def.FNK}. For every $\ell, \ell' \ge 1$, we write
\begin{equation*}  
R_0^{\ell,\ell'} := \frac{\al^\ell \wedge \al^{\ell'}}{K}, \qquad R_+^{\ell,\ell'} = (R^{\ell,\ell'}_{d,d'})_{1 \le d,d' \le D} = \frac{\si^\ell (\si^{\ell'})^*}{N}, \qquad R^{\ell,\ell'} = (R_0^{\ell,\ell'},R_+^{\ell,\ell'}),
\end{equation*}
and for every $n \ge 1$,
\begin{equation*}  
R^{\le n} := (R^{\ell,\ell'})_{1 \le \ell, \ell' \le n}, \qquad R :=  (R^{\ell,\ell'})_{\ell, \ell' \ge 1}.
\end{equation*}
For every $\ell, \ell' \ge 1$, we denote by $R^{\ell,\ell'}_{\mathrm{sym}}$ and $R^{\ell,\ell'}_{\mathrm{skew}}$ the symmetric and skew-symmetric parts of the matrix $R_+^{\ell,\ell'}$ respectively. That is, for every $d, d' \in \{1,\ldots, D\}$,
\begin{equation*}  
R^{\ell,\ell'}_{\mathrm{sym},d,d'} = \frac{R^{\ell,\ell'}_{d,d'} + R^{\ell,\ell'}_{d',d}}{2}  \quad \text{ and } \quad R^{\ell,\ell'}_{\mathrm{skew},d,d'}  = \frac{R^{\ell,\ell'}_{d,d'} - R^{\ell,\ell'}_{d',d}}{2}.
\end{equation*}
For every matrix $A \in \R^{D\times D}$ and integer $p \ge 1$, we denote by $A^{\odot p}$ its $p$-fold Schur product, that is, for every $d,d' \in \{1,\ldots D\}$,
\begin{equation*}  
(A^{\odot p})_{d,d'} = (A_{d,d'})^p. 
\end{equation*}
We denote by $(\lambda_n)_{n \in \N}$  an enumeration of the set $[0,1] \cap \Q$, and by $(a_n)_{n \in \N}$ an enumeration of the set of matrices in $S^D_+$ with norm bounded by $1$ and rational coefficients. For convenience, we impose that $\lambda_0 = 0$, $a_0 = 0$, and that $(a_1,\ldots, a_{\frac{D(D+1)}{2}})$ is a basis of $S^D$ (which we can fix explicitly).
\begin{proposition}[GG implies symmetrization and synchronization]
\label{p.synchr}
There exists a constant $C < \infty$ depending only on $D$ and, for every $\ep > 0$, an integer $h_+ \ge 1$ such that the following holds for every $N,K \ge 1$, $t \ge 0$, and $q \in \bar U_K$. Assume that, for every continuous function $f = f(R^{\le n})$ such that $\|f\|_{L^\infty} \le 1$ and $n, h_1, \ldots, h_4 \in \{0,\ldots, h_+\}$,  we have
\begin{align}
\label{e.gg.delta}
& \bigg| \E \la f ( R^{\le n} ) \Ll(a_{h_1} \cdot (R_+^{1,n+1})^{\odot h_2} + \lambda_{h_3} R_0^{1,n+1}\Rr)^{h_4} \ra  
\\
\notag
& \qquad - \frac 1 n \E \la f(R^{\le n})  \ra \E \la \Ll(a_{h_1} \cdot (R_+^{1,2})^{\odot h_2}+ \lambda_{h_3} R_0^{1,2}\Rr)^{h_4} \ra 
\\
\notag
& \qquad 
- \frac 1 n\sum_{\ell = 2}^n \E \la f(R^{\le n}) \Ll(a_{h_1} \cdot (R_+^{1,\ell})^{\odot h_2}+ \lambda_{h_3} R_0^{1,\ell}\Rr)^{h_4}  \ra \bigg| \le \frac 1 {h_+}.
\end{align}
We then have
\begin{equation}
\label{e.symmetrization}
\E \la \Ll| R^{1,2}_{\mathrm{skew}} \Rr|  \ra \le \ep
\end{equation}
and
\begin{equation}  
\label{e.synchro}
\E \la \Ll| R^{1,2}_+ - \E \la R^{1,2}_+ \vb R^{1,2}_0 \ra \Rr|^2 \ra \le \frac{C}{K} + \ep K^3.
\end{equation}
\end{proposition}
\begin{remark}  
The phrase ``for every  $N,K \ge 1$, $t \ge 0$, and $q \in \bar U_K$'' in the statement of Proposition~\ref{p.synchr} is a convenience employed here to avoid changing setting and notation, but can be replaced by something more general. Indeed, the statement really applies to any random Gibbs measure $\la \cdot \ra$ defined on $\mcl H_N^D$ for any Hilbert space $\mcl H_N$, provided that the support of the measure $\la \cdot \ra$ is contained in the unit ball, and that the law of $R_0^{1,2}$ is sufficiently ``spread out''; see \cite[Proposition~5.5]{bipartite} for a more precise statement. Another notational convenience employed here is that, as will be seen below, the proof of Proposition~\ref{p.synchr} only really uses \eqref{e.gg.delta} with~$h_2  \in \{1,2\}$ (and this is not the only case that appears in the assumption but is not actually used in the proof: for instance, in the case of $h_2 = 2$, we only use \eqref{e.gg.delta} with $h_3 = 0$).
\end{remark}
\begin{proof}[Proof of Proposition~\ref{p.synchr}]
For every $a \in S^D_+$ and $\ell, \ell' \in \mcl H_N^D$, we have
\begin{equation*}  
a \cdot R_+^{\ell,\ell'} = \frac{\sqrt{a} \, \si^{\ell}}{\sqrt{N}} \cdot \frac{\sqrt{a} \, \si^{\ell'}}{\sqrt{N}}.
\end{equation*}
We can thus appeal to \cite[Proposition~5.5]{bipartite} to obtain the existence of an integer $h_+ \in \N$ such that, under the assumption that \eqref{e.gg.delta} holds for every $n,h_1,\ldots, h_4 \in \{0,\ldots, h_+\}$, we have for every $h_1 \in \{0,\ldots, h_+\}$ that
\begin{equation*}  
\E \la \Ll( a_{h_1} \cdot R^{1,2}_+ - \E \la a_{h_1} \cdot R^{1,2}_+ \vb R^{1,2}_0 \ra \Rr)^2 \ra \le \frac{12}{K} + \ep K^3.
\end{equation*}
For $h_+ \ge \frac{D(D+1)}{2}$, the sequence $(a_{h_1})_{h_1 \in \{0,\ldots, h_+\}}$ contains a given basis of $S^D$. It thus follows that, for some constant $C < \infty$ that depends only on $D$,
\begin{equation*}  
\E \la \Ll| R^{1,2}_{\mathrm{sym}} - \E \la R^{1,2}_{\mathrm{sym}} \vb R^{1,2}_0 \ra \Rr|^2 \ra \le \frac{C}{K} + C \ep K^3.
\end{equation*}
This is \eqref{e.synchro} for the symmetric part of $R^{1,2}_+$, up to a redefinition of $\ep > 0$. It thus remains to show that \eqref{e.symmetrization} holds for $h_+$ sufficiently large. We argue by contradiction, assuming that no such $h_+$ exists. That is, we assume that there exists $\ep > 0$ and, for every  $h_+ \in \N$, a random Gibbs measure $\la \cdot \ra$ such that \eqref{e.symmetrization} is invalid. Up to the extraction of a subsequence, we can further assume that the overlap array $R$ converges in law under the measure $\E \la \cdot \ra$, in the sense of finite-dimensional distributions. We denote a limit overlap array by $\msf R$, defined with respect to a probability measure $\M$. In particular, we have that, for every $n,h_1,\ldots, h_4 \in \N$,
\begin{multline}
\label{e.gg}
 \M \Ll[f ( \msf R^{\le n} ) \Ll(a_{h_1} \cdot (\msf  R_+^{1,n+1})^{\odot h_2}\Rr)^{h_4} \Rr]
\\
 = \frac 1 n \M \Ll[f(\msf R^{\le n})  \Rr] \M \Ll[\Ll(a_{h_1} \cdot (\msf R_+^{1,2})^{\odot h_2}\Rr)^{h_4} \Rr] 
+ \frac 1 n\sum_{\ell = 2}^n \M \Ll[f(\msf R^{\le n})  \Ll(a_{h_1} \cdot (\msf R_+^{1,\ell})^{\odot h_2}\Rr)^{h_4}  \Rr] ,
\end{multline}
while
\begin{equation*}  
\M \Ll[ |\msf R^{1,2}_\mathrm{skew}| \Rr] \ge \ep.
\end{equation*}
We can then use the argument from the proof of \cite[Theorem~3]{pan.potts} to reach a contradiction. We recall this argument briefly here. Notice first that, by density and linearity, the identity \eqref{e.gg} holds for every $a_{h_1} \in S^D_+$. Using this relation with $h_2 = 1$ and arguing as in \cite{pan.multi} or \cite{bipartite}, we have that $\msf R^{1,2}_\mathrm{sym}$ is monotonically coupled. Indeed, this follows from the fact that for each $a,b \in S^D_+$, the hypothesis of \cite[Theorem~5.3]{bipartite} is satisfied with the quantities $R_1^{1,n+1}$ and $R_2^{1,n+1}$ appearing there replaced by $a \cdot R^{1,2}_\mathrm{sym}$ and $b \cdot R^{1,2}_\mathrm{sym}$ respectively, and with $\delta > 0$ arbitrary. The same holds for the symmetric part of the matrix $(\msf R^{1,2}_+)^{\odot 2}$, using \eqref{e.gg} with $h_2 = 2$. Using also Talagrand's positivity principle, see \cite[Theorem~2.16]{pan}, it thus follows that the laws of these matrices have representations as described in Proposition~\ref{p.uparrow.bij}. Notice also that for every $M \in \mcl M$ (with $\mcl M$ defined in Proposition~\ref{p.uparrow.bij}), there exists a function that allows us to calculate the value of $M(u)$ from the knowledge of its trace only; indeed, this follows from the fact that the function~$M$ cannot increase while keeping its trace constant. In particular, with probability one, the entire matrix $\msf R_\mathrm{sym}^{1,2}$ can be recovered from the knowledge of its trace only. The knowledge of the matrix $\msf R_\mathrm{sym}^{1,2}$ then allows us to compute the trace of the symmetric part of $(\msf R_+^{1,2})^{\odot 2}$, and therefore to also recover all of the entries of this symmetric matrix as well. These observations combine to ensure that, for each $d, d' \in \{1,\ldots D\}$, the quantities
\begin{equation*}  
\msf R^{1,2}_{d,d'} + \msf R^{1,2}_{d',d} \quad \text{ and } \quad (\msf R^{1,2}_{d,d'})^2 + (\msf R^{1,2}_{d',d})^2
\end{equation*}
can be inferred from the knowledge of $\tr(\msf R^{1,2}_\mathrm{sym})$. In particular, for a given value of $\tr(\msf R_+^{1,2})$ and for each $d,d' \in \{1,\ldots, D\}$, we can infer the identity of the set
\begin{equation*}  
\Ll\{ \msf R^{1,2}_{d,d'}, \msf R^{1,2}_{d',d} \Rr\}.
\end{equation*}
We aim to show that this set has cardinality one. Here and throughout the rest of this proof, we keep the indices $d,d' \in \{1,\ldots, D\}$ fixed. 

By \eqref{e.gg} and \cite{pan.aom}, we have that the array $\mathrm{Id} \cdot \msf R_+ = \tr (\msf R_+)$ is ultrametric. As a consequence, it satisfies the duplication property. That is, if $\msf t \in \R$ is a point in the support of $\tr (\msf R_\mathrm{sym}^{1,2})$, then for each integer $n \ge 1$, the support of $(\msf R_+^{\le n})$ contains a point $\td R^{\le n}$ such that, for every $\ell \neq \ell' \le n$,  we have $\tr ( \td R^{\ell,\ell'}) = \msf t$.
It thus follows that there exist $r, r' \in \R$ such that for this overlap array $\td R^{\le n}$ in the support of $(\msf R_+^{\le n})$, we have for every $\ell \neq \ell' \in \{1,\ldots, n\}$ that
\begin{equation*}  
\Ll\{ \td R^{\ell,\ell'}_{d,d'}, \td R^{\ell,\ell'}_{d',d} \Rr\} = \{r,r'\}.
\end{equation*}
Arguing by contradiction, suppose that $r \neq r'$. We can then construct a graph with vertex set $\{1,\ldots, n\}$ and, for each $\ell \neq \ell' \in \{1,\ldots, n\}$, draw an oriented edge from $\ell$ to~$\ell'$ if $\td R^{\ell,\ell'}_{d,d'} = r$ and $\td R^{\ell,\ell'}_{d',d} = r'$; and draw an oriented edge from $\ell'$ to $\ell$ otherwise. In this graph, there exist two disjoint subsets $V, V'$ of $\{1,\ldots, n\}$ of cardinality $s_n$, with $s_n \to \infty$ as $n$ tends to infinity, such that every edge between a point in $V$ and a point in $V'$ is oriented from $V$ to $V'$ (see for instance \cite{ehp}). By \cite[Theorem~1.7]{pan}, there exist variables $(\si^\ell)_{\ell \ge 1}$ taking values in the unit ball of  $\mcl H^D$ for some Hilbert space $\mcl H$ and such that for every $\ell \neq \ell'$, we have $\td{R}^{\ell,\ell'} = \si^\ell (\si^{\ell'})^*$. We set
\begin{equation*}  
b := \frac{1}{|V|} \sum_{\ell \in V} \si^\ell,  \qquad b' := \frac{1}{|V|} \sum_{\ell \in V'} \si^{\ell},
\end{equation*}
and consider
\begin{equation*}  
\Ll| b - b' \Rr|^2  = \frac 1 {|V|^2} \sum_{d'' = 1}^D \Ll|\sum_{\ell \in V} \si^\ell_{d''} - \sum_{\ell \in V'} \si^{\ell}_{d''}\Rr|^2 .
\end{equation*}
Expanding the square, and recalling that $\si^\ell_{d''} \cdot \si^{\ell'}_{d''}$ does not depend on $(\ell,\ell')$ provided that $\ell \neq \ell'$, we obtain the existence of a constant $C < \infty$ not depending on $n$ such that
\begin{equation*}  
\Ll| b-b'  \Rr|^2 \le \frac C {|V|}.
\end{equation*}
However, by the construction of $V$ and $V'$, we have $b_d \cdot b'_{d'} = r$ and $b'_{d} \cdot b_{d'} = r'$. Combining these observations yields that
 $|r-r'| \le 2(C/|V|)^\frac 1 2$. Since $n$ was arbitrary, and $|V| = s_n$ tends to infinity with $n$, we conclude that $r = r'$, as desired.
\end{proof}

We now proceed to introduce the random fields that we will add to the energy function as small perturbations. These perturbations will be used later to ensure the validity of the assumption in Proposition~\ref{p.synchr}. For every $h = (h_1,\ldots, h_4) \in \N^4$, let $(H^h_N(\si,\al))_{\si \in \mcl H_N^D, \al \in \N^k}$ be the centered Gaussian random field whose covariance is such that, for every $\si, \tau \in \mcl H_N^D$ and $\al,\be \in \N^k$, we have
\begin{equation}
\label{e.def.cov.HNh}
\E \Ll[ H_N^h(\si,\al) \, H_N^h(\tau,\be) \Rr] = N \Ll( a_{h_1} \cdot \Ll( \frac{\si \tau^*}{N} \Rr)^{\odot h_2} + \lambda_{h_3} \frac{\al \wedge \beta}{K}\Rr)^{h_4}.
\end{equation}
The existence of these random fields is shown in Section~\ref{s.examples}, see in particular \eqref{e.def.xi1} and Proposition~\ref{p.closure}. We assume that they are independent of each other, and independent of any other sources of randomness in the problem. For an integer $h_+ \in \N$ that will be chosen sufficiently large as a function of $K$ in the course of the proof, and setting $\td h_+ := (h_++1)^4 + \frac{D(D+1)}{2}$, we define our perturbation by setting, for every $x \in \R^{\td h_+}$, $\si \in \mcl H_N^D$ and $\al \in \N^k$,
\begin{equation}  
\label{e.def.HNx}
H_N^x(\si,\al) := N^{-\frac 1 {16}} \sum_{h \in \{0,\ldots, h_+\}^4}  x_{h} H_N^h(\si,\al)   + N^{-\frac 1 {16}} \sum_{h = 1}^{D(D+1)/2} x_{h} a_{h} \cdot {\si \si^*},
\end{equation}
where we use the following non-standard indexing convention for $x \in \R^{\td h_+}$:
\begin{equation*}  
x = \Ll((x_h)_{h \in \{0,\ldots, h_+\}^4},(x_h)_{h \in \{1,\ldots, D(D+1)/2\}}\Rr).
\end{equation*}
As has become clear, we allow ourselves to use essentially the same notation for $H^\mu_N$, $H^h_N$, and $H^x_N$, which are defined respectively in \eqref{e.def.Hmu}, \eqref{e.def.cov.HNh}, and \eqref{e.def.HNx}. This allows us to avoid heavier notation, and does not in fact create ambiguities, provided that the ``type'' of the ``exponent'' is known: the quantity $\mu$ in \eqref{e.def.Hmu} is a probability measure; the quantity $h$ in \eqref{e.def.cov.HNh} belongs to $\N^4$; and the quantity $x$ in \eqref{e.def.HNx} belongs to $\R^{\td{h}_+}$. The same remarks apply as well to the notation $x_h$, which refers to different things according to whether $h \in \{0,\ldots, h_+\}^4$ or $h \in \{1,\ldots, \frac{D(D+1)}{2}\}$. The prefactor $N^{-\frac 1 {16}}$ in \eqref{e.def.HNx} will ensure that this additional energy function does not contribute to the limit free energy. We now add this perturbative term to the definition in \eqref{e.def.FN}: that is, for every $t \ge 0$, $\mu$ as in \eqref{e.def.mu.de}, and $x \in \R^{\td{h}_+}$, we set
\begin{multline*}  
F_N(t,\mu,x) := -\frac 1 N \log \int \sum_{\al \in \N^K} \exp \bigg( \sqrt{2t} H_N(\si) - Nt\xi\Ll(\frac{\si \si^*}{N}\Rr) 
\\
+ H_N^\mu(\si,\al) - \si\cdot q_K \si + H_N^x(\si,\al)\bigg) \, v_\al \, \d P_N(\si),
\end{multline*}
as well as $\bar F_N(t,\mu,x) := \E \Ll[ F_N(t,\mu,x) \Rr]$. It is again convenient to use the same notation as that introduced in \eqref{e.def.FN} and \eqref{e.def.barFN} here. The properties of $F_N$ obtained in Section~\ref{s.enriched} are still valid with the new, extended definition of $F_N$, for any fixed value of $x$, so the formulas displayed there are still valid. Note also that we can always dispel any possible confusion as to the identity of the function we wish to refer to by writing, respectively, $(t,\mu) \mapsto F_N(t,\mu)$ and $(t,\mu,x) \mapsto F_N(t,\mu,x)$.

The next proposition gives an upper bound on the size of the fluctuations of $F_N$. 
\begin{proposition}[concentration property]
\label{p.concentration}
Let $K, h_+ \ge 1$ be integers, and for every $(t,q,x) \in \R_+ \times \bar U_K \times \R^{\td{h}_+}$, let
\begin{equation*}  
F_N^{(K)}(t,q,x) := F_N \Ll( t, \frac 1 K \sum_{k = 1}^K \de_{q_k}, x \Rr) , \qquad \bar F_N^{(K)}(t,q,x) := \E \Ll[ F_N(t,q,x) \Rr] . 
\end{equation*}
For every $M \in [0, \infty)$, $p \in [1,\infty)$, and $\ep > 0$, there exists a constant $C < \infty$ such that, letting
\begin{multline*}  
B_M := \bigg\{ (t,q,x) \in \R_+ \times \bar U_K \times \R^{\td{h}_+} \ : \  t \le M, \ |q_K| \le M, 
\\
\text{ and } \  \forall h \in \{0,\ldots, h_+\}^4 \cup \Ll\{1,\ldots, \frac{D(D+1)}{2}\Rr\}, |x_h| \le M \bigg\} ,
\end{multline*}
we have for every $N \ge 1$ that
\begin{equation*}  
\E \Ll[ \sup_{B_M} \Ll| F_N^{(K)} - \bar F_N^{(K)} \Rr|^p \Rr] ^\frac 1 p \le C N^{-\frac 1 2 + \ep}.
\end{equation*}
\end{proposition}
The proof of Proposition~\ref{p.concentration} is similar to that of \cite[Proposition~4.2]{bipartite} and relies on relatively classical techniques, so we will not provide further details here. We record for future use that
for every $h \in \{0,\ldots, h_+\}^4$, we have, by Gaussian integration by parts,
\begin{align*}  
\dr_{x_h} \bar F_N 
& = -N^{-1-\frac 1 {16}} \E \la H_N^h(\si,\al) \ra
\\
& = N^{-\frac  1 8} x_h \, \E \la \Ll( a_{h_1} \cdot \Ll( \frac{\si \si'^*}{N} \Rr)^{\odot h_2} + \lambda_{h_3} \frac{\al \wedge \al'}{K}\Rr)^{h_4} - \Ll( a_{h_1} \cdot \Ll( \frac{\si \si^*}{N} \Rr)^{\odot h_2} + \lambda_{h_3} \Rr)^{h_4} \ra,
\end{align*}
while, for every $h \in \{1,\ldots, \frac{D(D+1)}{2}\}$, 
\begin{equation*}  
\dr_{x_h} \bar F_N = -N^{-1 - \frac 1 {16}} \E \la a_{h} \cdot \si \si^* \ra.
\end{equation*}

We are now ready to prove the main result of this section.

\begin{proof}[Proof of Proposition~\ref{p.super}]

We fix $h_+$ sufficiently large that Proposition~\ref{p.synchr} holds with $\ep = K^{-4}$. Throughout the proof, we lighten the notation and write $F_N$ and $\bar F_N$ in place of $F_N^{(K)}$ and $\bar F_N^{(K)}$. With this convention in place, recall that the function appearing in the statement of Proposition~\ref{p.super} is $(t,q) \mapsto \bar F_N(t,q)$, and we will also use the ``extended'' function $(t,q,x) \mapsto \bar F_N(t,q,x)$ during the course of the proof. Let $f$ be a subsequential limit of $(t,q) \mapsto \bar F_N(t,q)$. For convenience, we also omit to denote the subsequence along which the convergence of $(t,q) \mapsto \bar F_N(t,q)$ to $f$ holds. 

Since our aim is to show that $f$ is a solution to \eqref{e.super}, we give ourselves $(t_\infty,q_\infty) \in (0,\infty) \times \bar U_K$, and a smooth function $\phi \in C^\infty((0,\infty) \times \bar U_K)$ such that $f-\phi$ has a local minimum at $(t_\infty,q_\infty)$. 
We will show that, for a constant $C_0 < \infty$ that depends only on $D$ and $\xi$, we have
\begin{equation}  
\label{e.wewant}
\Ll(\dr_t \phi - \H_K(\nabla \phi) \Rr)(t_\infty,q_\infty) \ge -\frac{C_0}{K} .
\end{equation}
The structure of the proof is similar in outline to that of \cite[Theorem~4.1]{bipartite}. Whenever the arguments are similar, we will thus simply recall this structure; we only provide details when the arguments differ. We denote by $C < \infty$ a constant that may depend on $D$, $\xi$, $K$, $h_+$ (itself already fixed in terms of $K$), $\phi$, $t_\infty$, and $q_\infty$. We write 
\begin{equation*}  
x_\infty = (1,\ldots, 1) \in \R^{\td{h}_+},
\end{equation*}
and set, for every $(t,q,x) \in (0,\infty) \times \bar U_K \times \R^{\td{h}_+}$, 
\begin{equation*}  
\td \phi(t,q,x) := \phi(t,q) - (t-t_\infty)^2 - |q-q_\infty|^2 - |x-x_\infty|^2. 
\end{equation*}
As in \cite{bipartite}, we can identify $(t_N,q_N,x_N) \in (0,\infty)\times \bar U_K \times \R^{\td{h}_+}$ which, for $N$ sufficiently large, is a local minimum of $\bar F_N - \td \phi$, and such that
\begin{equation}  
\label{e.lim.tn}
\lim_{N \to \infty} (t_N,q_N,x_N) = (t_\infty,q_\infty,x_\infty). 
\end{equation}
Using this property, one can show that
\begin{equation}  
\label{e.hessian.control}
-C \le \nabla_x^2 \bar F_N(t_N,q_N,x_N) \le 0
\end{equation}
(this follows from \cite[(4.25)]{bipartite}); and using  also Proposition~\ref{p.concentration}, that for every $\ep > 0$,
\begin{equation}
\label{e.grad.concentr}
\E \Ll[ \Ll| \nabla_x(F_N - \bar F_N)(t_N,q_N,x_N) \Rr| ^2   \Rr]  \le C N^{-\frac 1 2 + \ep},
\end{equation}
where the constant $C < \infty$ is now also allowed to depend on $\ep > 0$. These two estimates ensure the concentration of $H^h_N(\si,\al)$. Indeed, we start by writing, for every $h \in \{0,\ldots, h_+\}^4$,
\begin{equation*}  
\E\la \Ll(H_N^h - \E \la H_N^h \ra \Rr)^2\ra = \E\la \Ll(H_N^h - \la H_N^h \ra \Rr)^2\ra  + \E \la \Ll( \la H_N^h \ra - \E \la H_N^h \ra \Rr) ^2 \ra.
\end{equation*}
Moreover,
\begin{equation*}  
\dr^2_{x_h} \bar F_N = N^{-1 - \frac 1 8} \E\la \Ll(H_N^h - \la H_N^h \ra \Rr)^2\ra,
\end{equation*}
and
\begin{equation*}  
\dr_{x_h}(F_N - \bar F_N) = N^{-1-\frac 1 {16}} \Ll(\la H_N^h \ra - \E \la H_N^h \ra\Rr).
\end{equation*}
Combining this with \eqref{e.hessian.control} and \eqref{e.grad.concentr}, we thus get that, for every $h \in \{0,\ldots, h_+\}^4$,
\begin{equation}  
\label{e.concentr.HNh}
\E\la \Ll(H_N^h - \E \la H_N^h \ra \Rr)^2\ra \le C N^{\frac 3 2 + \frac 1 8 +\ep},
\end{equation}
where we implicitly understand that this relation is for the parameters $(t_N,q_N,x_N)$. 
The same reasoning also gives that, for every $h \in \Ll\{1,\ldots, \frac{D(D+1)}{2}\Rr\}$,
\begin{equation}  
\label{e.concentr.sisi}
\E\la \Ll(a_h \cdot \si \si^* - \E \la a_h \cdot \si \si^* \ra \Rr)^2\ra \le C N^{\frac 3 2 + \frac 1 8  +\ep}.
\end{equation}
It follows from \eqref{e.concentr.HNh} that, for every $n \in \N$ and function $g = g(R^{\le n})$ satisfying $\|g\|_{L^\infty} \le 1$, 
\begin{equation*}  
\Ll| \E \la g(R^{\le n}) H^h_N(\si,\al) \ra - \E \la g(R^{\le n}) H^h_N(\si,\al) \ra  \Rr| \le C N^{\frac 3 4 + \frac 1 {16} + \ep}.
\end{equation*}
A Gaussian integration by parts then essentially ensures the validity of \eqref{e.gg.delta}, with in fact a right-hand side that is replaced by $N^{-\frac 1 8 + \ep}$. The only difference is that there are some spurious ``self-overlap'' terms involving $\si \si^*$. These can be controlled using \eqref{e.concentr.sisi}. 

At this stage, we can then appeal to Proposition~\ref{p.synchr}, and recall our choice of $h_+$, to deduce that, for a constant $C_0 < \infty$ which only depends on $D$, we have
\begin{equation*}  
N^{-2} \E \la \Ll| \si \si'^* - \E \la \si \si'^* \vb \al \wedge \al' \ra \Rr|^2 \ra \le \frac{C_0}{K}.
\end{equation*}
(This is still understood to be for the parameters $(t_N,q_N,x_N)$.)
Allowing ourselves to enlarge the constant $C_0$ if necessary, and to let it depend on $\xi$ as well as $D$, we thus infer from \eqref{e.almost.equation} that
\begin{equation*}  
\Ll|\dr_t \bar F_N - \frac 1 K \sum_{k = 1}^K \xi \Ll( K \, \dr_{q_k} \bar F_N \Rr) \Rr|(t_N,q_N,x_N)\le \frac{C_0}{K} .
\end{equation*}
Since each coordinate of the positive semidefinite matrix $K \, \dr_{q_k} \bar F_N = N^{-2} \E \la \si \si'^* \vb \al \wedge \al' =k \ra$ is bounded by $1$, we may as well replace $\xi$ by $\bar \xi$ in the above display. Moreover, by Proposition~\ref{p.basic.ineq}, we have that $\nabla_q \bar F_N \in \bar U_K$. Recalling also \eqref{e.def.Hk}, we thus get that
\begin{equation}  
\label{e.eq.barfn}
\Ll|\dr_t \bar F_N - \H_K(\nabla_q \bar F_N) \Rr|(t_N,q_N,x_N)\le \frac{C_0}{K} .
\end{equation}
We now argue that
\begin{equation}
\label{e.diff.H.0}
\Ll(\H_K(\nabla_q \bar F_N) - \H_K(\nabla_q \td \phi)\Rr)(t_N,q_N,x_N) \ge 0.
\end{equation}
In view of \eqref{e.extend.H}, this would follow from 
\begin{equation}  
\label{e.haha}
\Ll(\nabla_q \bar F_N - \nabla_q \td \phi\Rr)(t_N,q_N,x_N) \in \bar U_K^*.
\end{equation}
Since $(t_N,q_N,x_N)$ is a local minimum of $\bar F_N - \td \phi$, we have for every $q' \in \bar U_K$ that
\begin{equation*}  
(q'-q_N) \cdot (\nabla_q \bar F_N - \nabla_q \td \phi)(t_N,q_N,x_N) \ge 0.
\end{equation*}
In particular, since $\bar U_K$ is a convex cone, we can substitute $(q'-q_N)$ by $q'$ in the display above. Recalling also \eqref{e.def.baruk*}, we obtain \eqref{e.haha}.

We can now combine \eqref{e.eq.barfn}, \eqref{e.diff.H.0}, and the observation that $\dr_t(\bar F_N - \td \phi)(t_N,q_N,x_N) = 0$, to obtain that
\begin{equation*}  
\Ll(\dr_t \td \phi - \H_K(\nabla_q \td \phi)\Rr) (t_N,q_N,x_N) \ge -\frac{C_0}{K} .
\end{equation*}
Since $\td \phi$ is smooth and $\H_K$ is continuous, see \eqref{e.lip.H}, we then use \eqref{e.lim.tn} to replace $(t_N,q_N,x_N)$ by $(t_\infty,q_\infty,x_\infty)$ in the above display. In view of the definition of $\td \phi$, this yields \eqref{e.wewant}.
\end{proof}

%
%
%
%
%
%

\section{Examples}
\label{s.examples}

In this section, we present examples of spin systems that satisfy the assumptions of the present paper. These include vector-valued spins, spins with multiple types, as well as systems in which a spin system is coupled with a random variable coming from a Poisson-Dirichlet cascade. The main point is to encode each of these situations in the form of \eqref{e.def.cov}, and to verify that the function $\xi$ appearing there is proper. We recall that a function $\xi : S^D_+ \to \R$ is said to be proper if it is increasing and for every $b \in S^D_+$, the mapping $a \mapsto \xi(a+b) - \xi(a)$ is increasing over~$S^D_+$. If $\xi$ is continuously differentiable, then this is equivalent to the statement that~$\nabla \xi$ takes values in $S^D_+$, and is increasing there. We will most of the time discuss functions $\xi$ defined on the whole set $\R^{D\times D}$ of $D$-by-$D$ matrices; in this case, we say that $\xi$ is proper if its restriction to $S^D_+$ is proper. After going through several examples and explaining how they fit into the present framework, we show in Proposition~\ref{p.charact} that, up to a regularity assumption on $\xi$, every function $\xi$ such that \eqref{e.def.cov} holds must be proper. In the last subsection, we explain how to build regularizations as defined in the paragraph preceding Theorem~\ref{t.main1}. 

\subsection{Vector-valued spins}

We fix $\mcl H_N = \R^N$, an integer $p \ge 2$, coefficients $\beta_1,\ldots, \beta_D \ge 0$, and for each $\si = (\si_{d,i})_{1 \le d \le D, 1 \le i \le N} \in \R^{D\times N}$, we set
\begin{equation*}  
H_N(\si):= N^{-\frac{p-1}{2}} \sum_{d = 1}^D \beta_{d} \, \sum_{1 \le i_1, \ldots, i_p \le N} J_{i_1,\ldots, i_p} \si_{d,i_1} \cdots \si_{d,i_p},
\end{equation*}
where $(J_{i_1,\ldots, i_p})_{1 \le i_1,\cdots, i_p \le N}$ are independent standard Gaussians. This model was treated in the case when $p$ is even in \cite{pan.vec}. One motivation for considering this model is the situation in which the vectors $\si_{1}$ and $\si_{2} \in \R^N$ are two ``replicas'' with some nontrivial coupling, and this motivates the choice of using the same random coupling variables $J_{i_1,\ldots, i_p}$ for all values of $d$. The following more general model could also be considered:
\begin{equation}  
\label{e.vector-valued}
H_N(\si):= N^{-\frac{p-1}{2}} \sum_{d = 1}^D  \sum_{1 \le i_1, \ldots, i_p \le N} J^{(d)}_{i_1,\ldots, i_p} \si_{d,i_1} \cdots \si_{d,i_p},
\end{equation}
where, for each $i_1,\ldots, i_p$, the Gaussian vector $(J^{(d)}_{i_1,\ldots, i_p})_{1 \le d \le D}$ has covariance $\msf C \in \R^{D\times D}$, and where these vectors are independent as we vary $i_1,\ldots, i_p$ (and we can absorb the parameters $\beta_d$ into the covariance matrix $\msf C$). For every $\si,\tau \in \R^{D\times N}$, we have
\begin{equation*}  
\E  \Ll[ H_N(\si) H_N(\tau) \Rr] = N \sum_{d,d' = 1}^D \msf C_{d,d'}  \Ll(\frac{\si_d \cdot  \tau_{d'}}{N}\Rr)^p.
\end{equation*}
Letting $\xi : \R^{D\times D} \to \R$ denote the function
\begin{equation*}  
\xi(A) := \sum_{d,d' = 1}^D C_{d,d'} A_{d,d'}^p
\end{equation*}
thus allows us to match the expression in \eqref{e.def.cov}. Denoting by $A^{\odot p}$ the $p$-fold Schur product of $A$, that is, for every $d,d' \in \{1,\ldots, D\}$,
\begin{equation*}  
(A^{\odot p})_{d,d'} :=(A_{d,d'})^p,
\end{equation*}
we can rewrite the function $\xi$ as
\begin{equation}  
\label{e.def.xi1}
\xi(A) = \msf C \cdot A^{\odot p}.
\end{equation}
We also write $A \odot B$ to denote the Schur product of two matrices $A,B \in \R^{D\times D}$, that is, for every $d,d' \in \{1,\ldots, D\}$, we set
\begin{equation*}  
(A \odot B)_{d,d'} := A_{d,d'} B_{d,d'}.
\end{equation*}
\begin{proposition}[Convexity criterion and monotonicity of $\nabla \xi$]
\label{p.xi.vector}
(1) For every $\msf C \in S^D_+$ and integer $p \ge 2$, the function $\xi$ defined in \eqref{e.def.xi1} is proper.

(2) The function $\xi$ is convex over $S^D_+$ if and only if one of the following two conditions hold:
\begin{itemize}  
\item[(a)] the integer $p$ is even and every entry of the matrix $\msf C$ is nonnegative;
\item[(b)] the integer $p$ is odd and the matrix $\msf C$ is diagonal.
\end{itemize}
\end{proposition}
\begin{proof}
The statement that $\nabla \xi$ is increasing over $S^D_+$ can be equivalently rewritten as: for every $A,B, E \in S^D_+$, 
\begin{equation}  
\label{e.increasing.grad1}
A \le B \quad \implies \quad \nabla \xi(A) \cdot E \le \nabla \xi(B) \cdot E. 
\end{equation}
%
%
For every $A,E \in \R^{D\times D}$, we have, as $\ep \to 0$,
\begin{equation}  
\label{e.taylor1}
\xi(A + \ep E) = \xi(A) +  p\msf C \cdot \Ll(\ep A^{\odot p-1} \odot E +  (p-1)\ep^2 A^{\odot p-2} \odot E^{\odot 2} \Rr)+ O(\ep^3).
\end{equation}
In particular,
\begin{equation*}  
\nabla \xi(A) \cdot E = p\msf C \cdot (A^{\odot p-1} \odot E). 
\end{equation*}
By the Schur product theorem, see e.g.\ \cite[Theorem~7.21]{matrix}, if $A,B \in S^D_+$ are such that $A \le B$, then $A^{\odot p-1} \le B^{\odot p-1}$. Using the Schur product theorem once more yields~\eqref{e.increasing.grad1}. The fact that $\nabla \xi$ maps $S^D_+$ into itself can be proved in the same way.

Coming back to \eqref{e.taylor1}, we see that the convexity of the function $\xi$ over $S^D_+$ is equivalent to the statement that, for every $A \in S^D_{++}$ and $E \in S^D$, we have 
\begin{equation}  
\label{e.convex.xi.condition}
\msf C \cdot \Ll( A^{\odot p-2} \odot E^{\odot 2}  \Rr) \ge 0.
\end{equation}
If $p$ is even and every entry of $\msf C$ is nonnegative, then \eqref{e.convex.xi.condition} is clearly satisfied. Irrespectively of the value of $p$, each diagonal element of $A$ is nonnegative, since $A \in S^D_+$. Hence, each diagonal element of the matrix $A^{\odot p-2} \odot E^{\odot 2}$ is nonnegative. Recalling also that $\msf C \in S^D_+$, we see that the condition \eqref{e.convex.xi.condition} is also satisfied whenever $\msf C$ is a diagonal matrix. 

Suppose now that some non-diagonal element of $\msf C$ is nonzero; without loss of generality we may assume that $\msf C_{1,2} \neq 0$, and denote $s :=- \frac{\msf C_{1,2}}{2|\msf C_{1,2}|} \in \{-\frac 1 2,\frac 1 2\}$. We choose matrices $A$ and $E$ that only have nonzero coordinates at indices in $\{1,2\}^2$, so that we can represent them as $2$-by-$2$ matrices. Consider
\begin{equation*}  
A := \begin{pmatrix}  
1 & s \\ s & 1 
\end{pmatrix} \in S^D_{++}, \quad \text{ and } \quad E := \begin{pmatrix}  
0 & 1 \\ 1&  0
\end{pmatrix} \in S^D.
\end{equation*}
Then 
\begin{equation*}  
A^{\odot p-2} \odot E^{\odot 2} = 
\begin{pmatrix}  
0 & s^{p-2} \\ s^{p-2} & 0
\end{pmatrix},
\end{equation*}
so that 
\begin{equation*}  
\msf C \cdot \Ll( A^{\odot p-2} \odot E^{\odot 2}  \Rr) = (-1)^{p-2} \frac{\msf C_{1,2}^{p-1}}{2^{p-1} |\msf C_{1,2}|^{p-2}}.
\end{equation*}
This violates \eqref{e.convex.xi.condition} whenever $p$ is odd, and also whenever $p$ is even and $\msf C_{1,2} < 0$.
\end{proof}
\begin{remark}  
Similar arguments allow us to build functions that are convex (and with null gradient at the origin) but not proper; examples can for instance be constructed by defining~$\xi$ as in \eqref{e.def.xi1}, for an even integer $p$ and a matrix $\msf C$ with nonnegative entries, but with $\msf C \notin S^D_+$. As shown in greater generality in Proposition~\ref{p.charact} below, for such functions, there does not exist any Gaussian random field $(H_N)$ such that \eqref{e.def.cov} holds.
\end{remark}

\subsection{Multiple types of spins}
We partition the set $\{1,\ldots, N\}$ into
\begin{equation*}  
\{1, \ldots, N\} = \bigcup_{d = 1}^D I_d,
\end{equation*}
where the subsets $(I_d)_{1 \le d \le D}$ are pairwise disjoint. Following \cite{pan.multi}, we would like to represent energy functions such as
\begin{equation}  
\label{e.multi.sk}
\frac{1}{\sqrt{N}}\sum_{i,j = 1}^N J_{ij} \si_i \si_j,
\end{equation}
where $(J_{ij})$ are independent Gaussians whose variance may depend on the identity of the indices $d,d'$ such that $i \in I_d$ and $j \in I_{d'}$. In order to fit this model into our framework, we reparametrize $\si$ as $\si = (\si_{d,i})_{1 \le d \le D, 1 \le i \le N} \in \R^{D \times N}$, and rewrite the energy function in \eqref{e.multi.sk} as
\begin{equation}  
\label{e.multitype}
H_N(\si) := \frac{1}{\sqrt{N}}\sum_{d_1,d_2 = 1}^D \sum_{i,j = 1}^N J_{ij}^{(d_1,d_2)} \si_{d_1,i} \si_{d_2,j},
\end{equation}
where $(J_{ij}^{(d_1,d_2)})_{1 \le d_1,d_2 \le D, 1 \le i,j \le N}$ are centered Gaussian random variables, and the $D^2$-dimensional vectors $((J_{ij}^{d_1,d_2})_{1 \le d_1,d_2 \le D})_{1 \le i,j \le N}$ are independent and identically distributed as the indices $i$ and $j$ vary. 
We denote by $\msf C \in S_+^{D^2 \times D^2}$ the covariance matrix of the vector $(J^{(d_1,d_2)}_{ij})_{1 \le d_1,d_2 \le D}$, that is, for every $d_1, d_2, d_1', d_2' \in \{1,\ldots, D\}$,
\begin{equation*}  
\msf C_{(d_1,d_2), (d_1',d_2')} := \E \Ll[ J_{ij}^{(d_1,d_2)} J_{ij}^{(d_1',d_2')} \Rr] .
\end{equation*}
To recover the model in \eqref{e.multi.sk}, we would impose that the matrix $\msf C$ be diagonal, and focus on reference measures $P_N$ such that with $P_N$-probability one, we have for every $d \in \{1,\ldots, D\}$ and $i \in \{1,\ldots, N\} \setminus I_d$ that $\si_{d,i} = 0$. 

For every $\si, \tau \in \R^{D\times N}$, we have
\begin{equation*}  
\E \Ll[H_N(\si) H_N(\tau)  \Rr] = \frac 1 N\sum_{d_1,d_2,d_1',d_2' = 1}^D \msf C_{(d_1,d_2), (d_1',d_2')}  (\si_{d_1} \cdot \tau_{d_1'}) (\si_{d_2} \cdot \tau_{d_2'}). 
\end{equation*}
This identity is of the form given in \eqref{e.def.cov}, provided that we set, for every $A \in \R^{D\times D}$,
\begin{equation}  
\label{e.def.xi.multi}
\xi(A) := \sum_{d_1,d_2,d_1',d_2' = 1}^D \msf C_{(d_1,d_2), (d_1',d_2')} A_{d_1,d_1'} \, A_{d_2,d_2'}.
\end{equation}
\begin{proposition}[Convexity criterion and monotonicity of $\xi$]
\label{p.xi.multi} (1) For every $\msf C \in S^{D^2}_+$, the function $\xi$ defined in \eqref{e.def.xi.multi} is proper.

(2) Let $\hat {\msf C} \in \R^{D^2 \times D^2}$ be obtained from $\msf C$ by rearranging its entries in such a way that, for every $d_1, d_2, d_1', d_2' \in \{1,\ldots, D\}$,
\begin{equation*}  
\hat {\msf C}_{(d_1,d_1'),(d_2,d_2')} := \msf C_{(d_1,d_2),(d_1',d_2')}.
\end{equation*}
The function $\xi$ is convex if and only if the symmetric part of the matrix $\hat{\msf C}$ is positive semidefinite. 
\end{proposition}
\begin{proof}
For every $A, B \in \R^{D\times D}$, we denote by $A \otimes B \in \R^{D^2 \times D^2}$ the matrix such that, for every $d_1, d_1', d_2, d_2' \in \{1,\ldots, D\}$ 
\begin{equation*}  
(A \otimes B)_{(d_1,d_2),(d_1',d_2')} = A_{d_1,d_1'} B_{d_2,d_2'}.
\end{equation*}
With this notation in place, we can write the function $\xi$ as, for every $A \in \R^{D\times D}$,
\begin{equation*}  
\xi(A) = \msf C \cdot (A \otimes A).
\end{equation*}
Notice that, for every $A, E \in \R^{D\times D}$, we have
\begin{equation*}  
\xi(A+\ep E) = \xi(A) + \ep \msf C \cdot (A \otimes E + E \otimes A) + O(\ep^2) \qquad (\ep \to 0).
\end{equation*}
As in the proof of Proposition~\ref{p.xi.vector}, in order to show that $\nabla \xi$ is increasing over $S^D_+$, it suffices to verify that, for every $A, B, E \in S^D_+$ with $A \le B$, we have
\begin{equation*}  
\nabla \xi(A) \cdot E \le \nabla \xi(B) \cdot E .
\end{equation*}
This is equivalent to
\begin{equation*}  
\msf C \cdot \Ll( (B-A) \otimes E + E \otimes (B-A) \Rr) \ge 0.
\end{equation*}
Recalling that the tensor product of two positive semidefinite matrices is positive semidefinite, see for instance \cite[Theorem~7.20]{matrix}, we obtain the result. The second part of the statement is immediate from the definition of $\xi$.  
\end{proof}
Notice that, whenever $\msf C$ is diagonal, the function $\xi$ actually only depends on the diagonal entries of its argument. In this case, it is at least heuristically reasonable to expect that the equation \eqref{e.hj} collapses to one that is posed over $\R^D_+$ only, instead of $S^D_+$ (since only diagonal elements enter into the equation). This is what was found in \cite{bipartite} in the case $D = 2$, and with the matrix $\msf C$ having only one nonzero entry at $\msf C_{(1,2),(1,2)}$. In this case, the matrix $\hat {\msf C}$ has only one nonzero entry, which is off the diagonal, so its symmetric part is not positive semidefinite.

\subsection{General tensor models}
The model in \eqref{e.multitype} is a generalization of the case $p = 2$ of the model introduced in \eqref{e.vector-valued}. (Relatedly, the Schur product $A \odot A$ is a submatrix of the tensor product $A \otimes A$.) One can also introduce a setting that would generalize the model in \eqref{e.vector-valued} for arbitrary values of the integer $p \ge 2$. We thus fix an integer $p \ge 2$, and let $(J^{(\mathbf d)}_{\mathbf i})_{\mathbf d \in \{1,\ldots D\}^p, \mathbf i \in \{1,\ldots, N\}^p}$ be centered Gaussian random variables such that the $D^p$-dimensional Gaussian vectors $((J^{(\mathbf d)}_{\mathbf i})_{\mathbf d \in \{1,\ldots, D\}^p})_{\mathbf{i} \in \{1,\ldots, N\}^p}$ are independent and identically distributed as the index $\mathbf{i}$ varies. We denote by $\msf C \in S^{D^p \times D^p}_+$ the covariance matrix of one of these vectors, so that for every $\mathbf{d}, \mathbf{d}' \in \{1,\ldots, D\}^p$ and $\mathbf{i}\in \{1,\ldots, N\}^p$,
\begin{equation*}  
\msf C_{\mathbf{d}, \mathbf{d}'} := \E \Ll[ J_{\mathbf{i}}^{(\mathsf d)}  J_{\mathbf{i}}^{(\mathsf d')} \Rr] .
\end{equation*}
We then set, for every $\sigma = (\sigma_{d,i})_{1 \le d \le D, 1 \le i \le N} \in \R^{D\times N}$,
\begin{equation*}  
H_N(\si) := N^{-\frac{p-1}{2}} \sum_{d_1, \ldots, d_p = 1}^D \ \sum_{i_1,\ldots, i_p = 1}^N J_{i_1, \ldots, i_p}^{(d_1, \ldots, d_p)} \si_{d_1,i_1} \cdots \, \si_{d_p,i_p}.
\end{equation*}
We have that \eqref{e.def.cov} holds for the function $\xi$ such that, for every $A \in \R^{D\times D}$,
\begin{equation}  
\label{e.def.xip}
\xi(A) = \msf C \cdot A^{\otimes p}. 
\end{equation}
\begin{proposition}
\label{p.xi}
For every $\msf C \in S^{D^p}_+$,  the function $\xi$ defined in \eqref{e.def.xip} is proper.
\end{proposition}
\begin{proof}
For every $A, E \in \R^{D\times D}$, we have, as $\ep$ tends to zero,
\begin{equation*}  
\xi(A+\ep E) = \xi(A) + \ep \msf C \cdot \Ll( A^{\otimes p-1} \otimes E + A^{\otimes p-2} \otimes E \otimes A + \cdots + E \otimes A^{\otimes p-1}  \Rr) + O(\ep^2). 
\end{equation*}
The conclusion follows as in the proof of Proposition~\ref{p.xi.multi}.
\end{proof}
The models investigated in \cite{acm20} are particular examples of this situation, with $p = 2$ and
\begin{equation*}  
\msf C_{\mathbf{d}, \mathbf{d}'} = \beta_{\min(d_1,d_2)}^2 \, \1_{\{\mathbf{d} = \mathbf{d}', \ |d_1-d_2| = 1\}}. 
\end{equation*}
In \cite{acm20}, the spin vectors $\si_1, \ldots, \si_D$ are thought of as having different total lengths. This can be encoded into the reference measure $P_N$.

\subsection{Poisson-Dirichlet variables}
Even for scalar models (called $p$-spin models), it is of interest to study the interplay between the spin variables and the Poisson-Dirichlet variables that are being added in the enriched model (for instance, one may want to understand the concentration of $\si \cdot \si'$ conditionally on $\al \wedge \al'$ taking a certain value). In this subsection, we explain how these Poisson-Dirichlet variables can be incorporated into the framework explored in the present paper. Concretely, recall the definition of $H^\mu_N(\si,\al)$ in \eqref{e.def.Hmu}, for a fixed choice of $\mu$ as in \eqref{e.def.mu.de}. For every $\si,\tau \in \mcl H_N^D$ and $\alpha, \beta \in \N^K$, we have 
\begin{equation}  
\label{e.hidden.bipartite}
\E \Ll[ H_N^\mu(\si,\al) H_N^\mu(\tau,\beta) \Rr] = 2    \si \cdot q_{\al \wedge \beta}\tau = 2 q_{\al \wedge \beta} \cdot \si \tau^*. 
\end{equation}
This is only one example of a natural energy function whose correlation features the quantity $q_{\al \wedge \be}$. In order to match \eqref{e.def.cov}, we would ideally want to represent this quantity as the matrix of scalar products of some variables in a Hilbert space. That is, for some Hilbert space $\hat{\mcl H}$, we would like to identify a mapping from $\N^K$ to ${\hat{\mcl H}}^D$, which we may denote by $\al \mapsto \hat \al$, such that $\hat \al \hat \beta^* = q_{\al \wedge \be}$, with $\hat \al \hat \beta^*$ as in \eqref{e.def.matrix}. Since we have been working with finite-dimensional Hilbert spaces so far, we will only realize such an identification for arbitrarily large but finite approximations of the set $\N^K$. We fix an integer $n \ge 1$, which we think of as being large, and let
\begin{equation*}  
\mcl A_n := \{0,\ldots, n\}^0 \cup \cdots \cup \{0,\ldots, n\}^K,
\end{equation*}
with the understanding that $\{0,\ldots, n\}^0 = \{\emptyset\}$. The set $\mcl A_n$ should be thought of as an approximation of the infinitary tree of depth $K$, denoted $\mcl A$, that was introduced in \eqref{e.def.mclA}. We denote by $\mcl L_n = \{0,\ldots, n\}^K$ the set of leaves of $\mcl A_n$, and also use the notation $\al_{|k}$ introduced in \eqref{e.def.truncate} for elements $\al \in \mcl L_n$. Let $(f_\al)_{\al \in \mcl A_n}$ be an orthonormal basis of $\R^{|\mcl A_n|}$, and for each $\al \in \mcl L_n$, let $\hat \al$ be the element of the tensor product $\R^{D\times D} \otimes \R^{|\mcl A_n|}$ defined by
\begin{equation}  
\label{e.def.hat.alpha}
\hat \al := \sum_{k = 0}^K (q_k - q_{k-1})^\frac 1 2 \otimes f_{\al_{|k}} . 
\end{equation}
In order to match the setting of \eqref{e.def.cov}, we can view $\R^{D\times D} \otimes \R^{|\mcl A_n|}$ as a $D$-fold Cartesian product:
\begin{equation*}  
\R^{D\times D} \otimes \R^{|\mcl A_n|} \simeq \R^D \otimes \R^D \otimes \R^{|\mcl A_n|} \simeq (\R^D \otimes \R^{|\mcl A_n|})^D.
\end{equation*}
Explicitly, writing $(e_{d})_{1 \le d \le D}$ for the canonical basis of $\R^D$, we realize the identification above through the mapping
\begin{equation*}  
\sum_{\al \in \mcl A_n} \sum_{d,d' = 1}^D a^{\al}_{dd'} \, e_d \otimes e_{d'} \otimes f_\al  \mapsto  \Ll( \sum_{\al \in \A_n} \sum_{d' = 1}^D a^\al_{1d'} e_{d'} \otimes f_\al, \ldots, \sum_{\al \in \A_n} \sum_{d' = 1}^D a^\al_{Dd'} e_{d'} \otimes f_\al \Rr) .
\end{equation*}
Let $a, b \in \R^{D \times D} \otimes \R^{|\mcl A_n|}$ have the decompositions
\begin{equation*}  
a = \sum_{\al \in \mcl A_n} \sum_{d,d' = 1}^D a^{\al}_{dd'} \, e_d \otimes e_{d'} \otimes f_\al \quad \text{ and } \quad b = \sum_{\al \in \mcl A_n} \sum_{d,d' = 1}^D b^{\al}_{dd'} \, e_d \otimes e_{d'} \otimes f_\al.
\end{equation*}
Using the identification above to rewrite $a, b$ as $(a_d)_{1 \le d \le D}$, $(b_d)_{1 \le d \le D} \in  (\R^D \otimes \R^{|\mcl A_n|})^D$, we have, for every $d, d' \in \{1,\ldots, D\}$,
\begin{equation*}  
a_d \cdot b_{d'} = \sum_{\al \in \mcl A_n} \sum_{d'' = 1}^D a^\al_{dd''} b^\al_{d'd''}.
\end{equation*}
If moreover $a,b \in S^D \otimes \R^{|\mcl A_n|}$, then, using the notation in \eqref{e.def.matrix},
\begin{equation*}  
a b^* =  \sum_{\al \in \mcl A_n} a^\al b^\al.
\end{equation*}
In particular, in view of \eqref{e.def.hat.alpha}, we have for every $\al,\beta \in \mcl L_n$ that
\begin{equation*}  
\hat{\al} \hat\beta^* = q_{\al \wedge \beta},
\end{equation*}
as desired.
In particular, the right side of \eqref{e.hidden.bipartite} can now be seen as having the same bipartite structure as that investigated in \cite{bipartite}. 

Let us write $\hat{\mcl H}_n := \R^{D} \otimes \R^{|\mcl A_n|}$. In expressions such as \eqref{e.def.cov}, we understood that the energy function was defined over the entire Hilbert space, now $\mcl H_N^D \times \hat{\mcl H}_n^D$, while so far we have only made sense of the energy function over a subset of this Hilbert space.  The framework could be modified to require the energy function to be defined only on the support of the measure of interest; however, it is simpler to indeed extend the energy function $(\si,\al) \mapsto H_N^\mu(\si,\al)$: we give ourselves a standard Gaussian vector $J$ over the tensor product $\mcl H_N \otimes \hat{\mcl H}_n$, and  set, for every $\si \in \mcl H_N^D$ and $a \in \hat{\mcl H}_n^D$, 
\begin{equation*}  
H_N'(\si,a) :=  J \cdot \sum_{d = 1}^D \, \si_d \otimes a_{d}.
\end{equation*}
We then have, for every $\si,\tau \in \mcl H_N^D$ and $a,b \in \hat{\mcl H}_n^D$,
\begin{equation*}  
\E \Ll[ H_N'(\si,a) H_N'(\tau,b) \Rr]  = \sum_{d,d' = 1}^D (a_d \cdot b_{d'}) (\si_d \cdot \tau_{d'})
 = (ab^*)\cdot (\sigma \tau^*).
\end{equation*}
We have thus defined a random energy function $H_N'$ over the entire space $\mcl H_N^D \times {\hat{\mcl H}}_n^D$, and for every $\al, \beta \in \mcl L_n$, we have
\begin{equation*}  
\E \Ll[ H_N'(\si,\hat \al) H_N'(\tau,\hat \be) \Rr] = q_{\al\wedge \beta} \cdot (\si \tau^*).
\end{equation*}
The difference between this model and the one investigated in \cite{bipartite} is contained in the choice of the underlying reference measure for the variables taking values in $\hat{\mcl H}_n^D$. We would want this measure to be a truncated version of 
\begin{equation*}  
\sum_{\al \in \N^k} v_\al \de_{\hat \al}, 
\end{equation*}
for instance
\begin{equation*}  
Q_n := \sum_{\al \in \N^k} \Ll(\1_{\{\al \in \mcl L_n\}} v_\al \de_{\hat \al} + \1_{\{\al \notin \mcl L_n\}} v_\al \de_{0}\Rr) \in \mcl P(\hat{\mcl H}_n^D).
\end{equation*}
We can then let $n$ diverge to infinity with $N$ to get an asymptotically exact description of the model of interest. For instance, we have indeed that the Gibbs measure over $\mcl H_N^D \times \hat{\mcl H}_n^D$ proportional to
\begin{equation*}  
\1_{\{ a \neq 0 \}} \exp \Ll( H_N(\si) + H_N'(\si,a) \Rr) \, \d P_N(\si) \, \d Q_n(a)
\end{equation*}
is the image of the measure over $\mcl H_N^D \times \mcl L_n$ proportional to
\begin{equation}  
\label{e.def.illustr}
\sum_{\al \in \mcl L_n} \exp \Ll(  H_N(\si) + H_N^\mu(\si,\alpha) \Rr)  \, \d P_N(\si) \, v_\al  \de_\al
\end{equation}
under the mapping
\begin{equation*}  
\Ll\{
\begin{array}{rcl}  
\mcl H_N^D \times \mcl L_n  & \to & \mcl H_N^D \times \hat{\mcl H}_n^D \\
(\si,\al) & \mapsto & (\si,\hat \al).
\end{array}
\Rr.
\end{equation*}
The measure in \eqref{e.def.illustr} was only chosen for illustration; the same consideration applies to the measure appearing in \eqref{e.def.Gibbs} for instance (in fact, this measure is of the form of \eqref{e.def.illustr} after a change of the measure $P_N$). 

\subsection{Closure properties of covariance functions}
In this subsection, we first discuss the closure properties of the space of functions $\xi$ that satisfy \eqref{e.def.cov} for some random energy function $H_N$. It is straightforward to verify that the space of functions $\xi$ that satisfy~\eqref{e.def.cov} for some $H_N$ is a convex cone: that is, if $\xi_1$ and~$\xi_2$ are two functions in this space, and if $\al, \be \ge 0$, then the function $\al \xi_1 + \be \xi_2$ also belongs to this space. This space is also closed under multiplication, as shown in the following proposition.
\begin{proposition}[Closure under multiplication]
\label{p.closure}
Let $(H_1(\sigma))_{\sigma \in \mcl H}$ and $(H_2(\sigma))_{\sigma \in \mcl H}$ be two centered Gaussian fields defined over the same Hilbert space $\mcl H$. There exists a centered Gaussian field $(H(\si))_{\si \in \mcl H}$ whose covariance is the product of the covariances of~$H_1$ and~$H_2$: for every $\si,\tau  \in \mcl H$,
\begin{equation}  
\label{e.cov.prod}
\E \Ll[ H(\si) H(\tau)\Rr] = \E \Ll[ H_1(\si) H_1(\tau)\Rr]  \, \E \Ll[ H_2(\si) H_2(\tau)\Rr] .
\end{equation}
\end{proposition}
\begin{proof}
By Kolmogorov's extension theorem, it suffices to justify, for every finite subset $\mcl S$ of $\mcl H$, the existence of a centered Gaussian field satisfying \eqref{e.cov.prod} for every $\si,\tau \in \mcl S$. This in turn amounts to the verification of the statement that the matrix
\begin{equation*}  
\Ll( \E \Ll[ H_1(\si) H_1(\tau)\Rr]  \, \E \Ll[ H_2(\si) H_2(\tau)\Rr]  \Rr) _{\si,\tau \in \mcl S}
\end{equation*}
is positive semidefinite. Since, for every $a \in \{1, 2\}$, the matrix $\Ll( \E \Ll[ H_a(\si) H_a(\tau)\Rr]\Rr) _{\si,\tau \in \mcl S}$ is positive semidefinite, the desired result follows from the Schur product theorem. 
\end{proof}
In principle, these observations (stability under positive linear combinations and multiplications) allow us to generate more examples of random fields whose covariance function can be put in the form displayed in \eqref{e.def.cov}; for instance, the existence of a random field with covariance given by \eqref{e.cov.prod} becomes clear. However, except by the obvious operation of taking positive linear combinations, we cannot generate truly new random fields by proceeding in this way. Indeed, the general form \eqref{e.def.xip}, with $p \in \N$ and $\msf C \in S^{D^p}_+$, encompasses all other examples discussed in the previous subsections; and for every $p, p' \in \N$, $\msf C \in S^{D^{p}_+}$, $\msf C' \in S^{D^{p'}}_+$, and $A \in \R^{D\times D}$, we have
\begin{equation*}  
\Ll(\msf C \cdot A^{\otimes p}\Rr) \Ll(\msf C' \cdot A^{\otimes p'}\Rr) = \Ll( \msf C \otimes \msf C' \Rr) \cdot  A^{\otimes (p+p')},
\end{equation*}
with $\msf C \otimes \msf C' \in S^{D^{p+p'}}_+$. So far the most general functions $\xi$ that we can construct are therefore of the form
\begin{equation*}  
\xi(A) = \sum_{p = 0}^{+\infty} \msf C^{(p)} \cdot A^{\otimes p},
\end{equation*}
where $\msf C^{(p)} \in S^{D^p}_+$, and the norm of $\msf C^{(p)}$ decays sufficiently fast (faster than any negative exponential of $p$ would do) as $p$ tends to infinity. The next proposition provides with some evidence that there cannot be many more examples.

\begin{proposition}[characterization of admissible functions]
\label{p.charact}
Let $D \ge 1$ be an integer, $\mcl H$ be a Hilbert space, $(H(\si))_{\si \in \mcl H^D}$ be a centered Gaussian field, and $\xi : \R^{D\times D} \to \R$ be a function such that, for every $\si,\tau \in \mcl H^D$,
\begin{equation}  
\label{e.def.cov.H}
\E \Ll[ H(\si) H(\tau) \Rr] = \xi(\si \tau^*).
\end{equation}
If $\xi$ admits an absolutely convergent power series expansion, then there exists a sequence of matrices $(\msf C^{(p)})_{p \in \N}$, with $\msf C^{(p)} \in S^{D^p}_+$ for every $p \ge 1$, such that for every $\si, \tau \in \mcl H^D$, 
\begin{equation}  
\label{e.charact}
\xi(\si \tau^*) = \sum_{p = 0}^{+\infty} \msf C^{(p)} \cdot (\si \tau^*)^{\otimes p}. 
\end{equation}
\end{proposition}
\begin{remark}  
\label{r.charact}
If the Hilbert space $\mcl H$ has dimension less than $D$, then the statement of~\eqref{e.charact} does not fully determine the function $\xi$. However, since we only ever want to refer to functions $\xi$ as they appear in \eqref{e.def.cov.H}, this is irrelevant, and we may modify $\xi$ outside of the set $\{\si \tau^*, \ \si,\tau \in \mcl H^D\}$ so that \eqref{e.charact} holds with $\si \tau^*$ replaced by any matrix $A \in \R^{D\times D}$. Once this is done, the possibly modified function $\xi$ is proper, by Proposition~\ref{p.xi}. 
\end{remark}
\begin{proof}[Proof of Proposition~\ref{p.charact}]
Without loss of generality, we may assume that the space $\mcl H$ is finite-dimensional. The statement is obvious if $\mcl H = \{0\}$. Otherwise, we may identify $\mcl H$ with $\R^I$ for some integer $I \ge 1$, and index every element $\si$ of $\mcl H^D$ as $\si =  (\si_d)_{1 \le d \le D} = (\si_{d,i})_{1 \le d \le D, 1 \le i \le I}$. 
Differentiating the relation \eqref{e.def.cov.H}, and using that $\xi$ has a power expansion and Kolmogorov's continuity theorem,  we can choose of modification of the mapping $\si \mapsto H(\si)$ that is $C^\infty$; and in fact, we then have that the mapping $\si \mapsto H_p(\si)$ can be written as a power series. More precisely, for every $p \in \N$ and $d_1,\ldots, d_p \in \{1,\ldots, D\}$, letting
\begin{equation*}  
J^{(d_1,\ldots, d_p)}_{i_1,\ldots, i_p} := \frac{1}{p!}\, \dr^p_{\si_{d_1,i_1}\cdots \, \si_{d_p,i_p}} H_p(0),
\end{equation*}
we have, for every $\si \in \R^D$,
\begin{equation*}  
H_N(\si) = \sum_{p = 0}^\infty \, \sum_{d_1,\ldots, d_p = 1}^D \sum_{i_1, \ldots, i_p = 1}^I J^{(d_1,\ldots, d_p)}_{i_1,\ldots, i_p} \, \si_{d_1,i_1} \cdots \, \si_{d_p,i_p}.
\end{equation*}
Notice that, for every $p \in \N$, $d_1,\ldots, d_p \in \{1,\ldots, D\}$, and $i_1,\ldots, i_p \in \{1,\ldots, I\}$,  we have uniformly over $|\si| \le 1$ that as $\tau$ tends to $0$,
\begin{equation*}  
\dr^p_{\si_{d_1, i_1} \cdots \, \si_{d_p,i_p}}\xi(\si \tau^*) = O(|\tau|^p),
\end{equation*}
and as a consequence, for every $p , q \in \N$ with $q < p$, $d_1,\ldots, d_p, d_1',\ldots, d_q' \in \{1,\ldots, D\}$, and $i_1,\ldots, i_p, i'_1,\ldots, i'_{q} \in \{1,\ldots, I\}$, we have 
\begin{equation*}  
\dr^p_{\si_{d_1,i_1} \cdots \, \si_{d_p,i_p}} \dr^q_{\tau_{d'_1,i'_1} \cdots \, \tau_{d'_q,i'_q}} \xi(0) = 0.
\end{equation*}
(We implicitly understand that it is the function $(\si, \tau) \mapsto \xi(\si \tau^*)$ that is being differentiated.) By differentiation of \eqref{e.def.cov.H}, we thus obtain that for every $p , q \in \N$ with $q < p$ and $d_1,\ldots, d_p, d_1',\ldots, d_q' \in \{1,\ldots, D\}$, the vectors $J^{(d_1,\ldots, d_p)}$ and $J^{(d'_1,\ldots, d'_q)}$ are uncorrelated, and therefore independent, since the family of all the $J$ variables is jointly Gaussian. A similar reasoning also yields the independence of the components of the vector $J^{(d_1,\ldots, d_p)}$. Denoting by $\msf C^{(p)} \in S^{D^p}_+$ the covariance matrix of the vector $(J^{(d_1,\ldots, d_p)})_{d_1,\ldots, d_p \in \{1,\ldots, D\}}$ , we thus obtain \eqref{e.charact}.
\end{proof}

\subsection{Construction of a regularization}
\label{ss.regularization}
In this subsection, we briefly explain how to build a regularization of a proper function.
\begin{proposition}[Construction of a regularization]
\label{p.construct.regul}
Every locally Lipschitz and proper function $\xi : S^D_+ \to \R$ admits a regularization.
\end{proposition}
\begin{proof}
For every $r > 0$, we write
\begin{equation*}  
\mcl B(r) := \Ll\{ a \in S^D_+ \ : \ \tr(a) \le r \Rr\} .
\end{equation*}
Notice that all positive semidefinite matrices with entries in $[-1,1]$ belong to $\mcl B(D)$. 
For every $a \in S^D_+$, we denote by $|a|_{\infty}$ the largest eigenvalue of $a$. We let
\begin{equation*}  
L := \| \, |\nabla \xi|_\infty \, \|_{L^\infty(\mcl B(2D))},
\end{equation*}
and for every $a \in S^D_+$,
\begin{equation*}  
\bar \xi(a) := 
\Ll|
\begin{array}{ll}  
\max \Ll( \xi(a), \xi(0) + 2L (\tr(a) - D) \Rr) & \text{ if } \tr(a) \le 2D, \\
\xi(0) + 2L(\tr(a) - D) & \text{ if } \tr(a) > 2D.
\end{array}
\Rr.
\end{equation*}
For every $a \in \mcl B(2D)$, we have
\begin{equation*}  
\xi(0) \le \xi(a) \le \xi(0) +  L \tr(a),
\end{equation*}
and thus $\bar \xi$ and $\xi$ coincide on $\mcl B(D)$, and $\bar \xi$ is a uniformly Lipschitz function. Its gradient takes values in $S^D_+$ almost everywhere, and is increasing. This shows that $\bar \xi$ is proper.
\end{proof}

\medskip

\noindent \textbf{Acknowledgements.} I was partially supported by the NSF grant DMS-1954357.

\small
\bibliographystyle{abbrv}
\bibliography{genbound}

\end{document}